\documentclass[10.2pt]{article}
\newif\ifPDF
\ifx\pdfoutput\undefined\PDFfalse
\else \ifnum \pdfoutput > 0 \PDFtrue
        \else \PDFfalse
        \fi
\fi

\usepackage[centertags]{amsmath}
\usepackage{amsfonts}
\usepackage{mathrsfs}
\usepackage{textcomp}
\usepackage{amssymb}
\usepackage{amsthm}
\usepackage{newlfont}
\usepackage[all]{xy}

\ifPDF
  \usepackage[pdftex]{color, graphicx}
  \usepackage[pdftex, bookmarks, colorlinks]{hyperref}
  \hypersetup{colorlinks=false}

\else
  \usepackage[dvips]{graphicx}
  \usepackage[dvips]{hyperref}
\fi

\usepackage[scale=0.80]{geometry}

\newtheorem{thm}{Theorem}[section]
\newtheorem{cor}[thm]{Corollary}
\newtheorem{lem}[thm]{Lemma}
\newtheorem{prop}[thm]{Proposition}
\theoremstyle{definition}
\newtheorem{defn}[thm]{Definition}
\newtheorem{rem}[thm]{Remark}

\numberwithin{equation}{section}
\newtheorem{NN}[thm]{}
\newcommand{\norm}[1]{\left\Vert#1\right\Vert}
\newcommand{\abs}[1]{\left\vert#1\right\vert}

\newcommand{\Int}{\mathbb Z}
\newcommand{\Comp}{\mathbb C}
\newcommand{\Ratn}{\mathbb Q}
\newcommand{\ep}{\varepsilon}

\newcommand{\F}{\mathcal{F}}
\newcommand{\Kzero}{{K}_0}
\newcommand{\Kone}{{K}_1}

\newcommand{\aff}{\mathrm{Aff}}
\newcommand{\T}{\mathbb T}
\newcommand{\CA}{C*-algebra}
\newcommand{\dt}{\delta}
\newcommand{\beq}{\begin{eqnarray}}
\newcommand{\eneq}{\end{eqnarray}}
\newcommand{\tforal}{\,\,\,\text{for\,\,\,all}\,\,\,}

\newcommand{\andeqn}{\,\,\,{\mathrm{and}}\,\,\,}

\newcommand{\rrm}{\mathrm}
\newcommand{\p}{\mathtt{p}}
\newcommand{\Z}{\mathbb Z}

\title{Lifting KK-elements, asymptotical unitary equivalence and classification
of simple C*-algebras
}
\author{Huaxin Lin and Zhuang Niu}

\date{\today}

\begin{document}
\maketitle

\begin{abstract}
Let $A$  and $C$ be two unital  simple  \CA s with tracial rank
zero. Suppose that $C$ is amenable and satisfies the Universal
Coefficient Theorem. Denote by ${{KK}}_e(C,A)^{++}$ the set of those $\kappa$ for which
$\kappa(K_0(C)_+\setminus\{0\})\subset K_0(A)_+\setminus\{0\}$ and $\kappa([1_C])=[1_A]$.
Suppose that $\kappa\in {KK}_e(C,A)^{++}.$
 We show
that there is a unital monomorphism $\phi: C\to A$ such that
$[\phi]=\kappa.$   Suppose that $C$ is a unital AH-algebra and
$\lambda: \mathrm{T}(A)\to \mathrm{T}_{\mathtt{f}}(C)$ is a
continuous affine map for which
$\tau(\kappa([p]))=\lambda(\tau)(p)$ for all projections $p$ in
all matrix algebras of $C$ and any $\tau\in \mathrm{T}(A),$
where $\mathrm{T}(A)$ is the simplex of tracial states of $A$ and
$\mathrm{T}_{\mathtt{f}}(C)$ is the convex set of faithful tracial states
of $C.$ We prove that there is a unital monomorphism $\phi: C\to
A$ such that $\phi$ induces both $\kappa$ and $\lambda.$

Suppose that $h: C\to A$ is a unital
monomorphism and  $\gamma \in \mathrm{Hom}(\Kone(C), \aff(A)).$ We
show that there exists a unital monomorphism $\phi: C\to A$ such
that $[\phi]=[h]$ in ${KK}(C,A),$ $\tau\circ \phi=\tau\circ h$ for
all tracial states $\tau$ and the associated rotation map can be
given by $\gamma.$
Denote by $KKT(C,A)^{++}$ the set of compatible pairs $(\kappa, \lambda),$ where
$\kappa\in KL_e(C,A)^{++}$ and $\lambda$ is a continuous affine map from $\mathrm{T}(A)$ to $\mathrm{T}_{\mathtt{f}}(C).$
Together with a result of asymptotic unitary
equivalence in \cite{Lin-Asy}, this provides a bijection from
the asymptotic unitary equivalence classes of unital monomorphisms from $C$ to $A$ to
$(KKT(C,A)^{++}, \mathrm{Hom}(K_1(C), {\rrm{Aff}}(\mathrm{T}(A)))/{\cal R}_0),$ where ${\cal R}_0$ is
a subgroup related to vanishing rotation maps.

As an application, combining with a result
of W. Winter (\cite{Winter-Z}),  we show
that two unital amenable simple ${\cal Z}$-stable \CA s are
isomorphic if they have the same Elliott invariant and the tensor
products of these \CA s with any UHF-algebras have tracial rank
zero. In particular, if $A$ and $B$ are two unital separable
simple ${\cal Z}$-stable \CA s which are inductive limits of \CA s of type
I with unique tracial states, then they are isomorphic if
and only if they have isomorphic Elliott invariant.

\end{abstract}

\section{Introduction}

Let $A$ and $B$ be two unital separable amenable simple \CA s
satisfying the Universal Coefficient Theorem.  It has been shown
(see \cite{Lnduke} and \cite{LinTAI}, also \cite{EG-RR0AH} and
\cite{EGL-AH}) that, if in addition, $A$ and $B$ have tracial rank
one or zero, then $A$ and $B$ are isomorphic if their Elliott
invariant are isomorphic. There are interesting simple amenable
\CA s  with stable rank one which do not have tracial rank zero or
one and some classification theorem have been established too (see
for example \cite{JS-Z}, \cite{JSsplit}, \cite{Niu-thesis} and \cite{EN-Tapprox}). One of the interesting classes of simple
amenable \CA s which satisfy the UCT are those simple ASH-algebras (approximate sub-homogeneous \CA s), or even more
general, those simple \CA s which are inductive limits of type I
\CA s. In the case of unital ${\cal Z}$-stable simple
ASH-algebras, if in addition, their projections separate
traces, a classification theorem can be  given (see
\cite{Winter-Z} and \cite{Lin-App}). More precisely, let $A$ and
$B$ be two unital simple ${\cal Z}$-stable ASH-algebras whose
projections separate the traces. Then $A\cong B$  if and only if
$({K}_0(A), {K}_0(A)_+, [1_A], {K}_1(A))$ is isomorphic to $({K}_0(B), {K}_0(B)_+, [1_B], {K}_1(B))$
provided that ${K}_i(A)$ are
finitely generated, or ${K}_i(A)$ contains its torsion part as a direct
summand. It should be noted that there are known examples that
ASH-algebras whose projections separate the traces but have real
rank other than zero. In particular, $A$ and $B$ may not have any
non-trivial projections as long as {each one has a unique tracial
state} (for instance, the Jiang-Su algebra $\mathcal Z$). It is clearly important to remove the above mentioned
restriction on $K$-theory.  The original purpose of this paper was
to remove these restrictions.

One of the technical tools used in the proof of \cite{Winter-Z} and \cite{Lin-App} is a theorem which determines when
   two unital monomorphisms from a unital AH-algebra $C$ to a unital simple \CA\, $A$
   are  asymptotically  unitarily  equivalent.
   In the case that $A$ has tracial rank zero, it was shown in \cite{Lin-Asy} that
   two such monomorphisms are asymptotically unitarily equivalent if they induce  the
   same element in $KK(C,A)$ and the same affine map on the tracial state
   spaces, and the rotation map associated with these two monomorphisms vanishes.
   It is equally important to determine the range
   of asymptotical unitary equivalence classes of those monomorphisms. In fact,
   the above mentioned restriction
   can be removed if the range can be determined.

   The first question we need to answer is the following: Let $A$ and $C$ be two unital
   simple amenable \CA s with tracial rank zero and let $C$  satisfy
   the UCT. Suppose that $\kappa{\in} {KK}_e(C,A)^{++}$ ({ i.e.,} those $\kappa\in {KK}(C,A)$ such that
  $\kappa({K}_0(C)_+\setminus \{0\})\subseteq  {K}_0(A)_+\setminus \{0\}$  with $\kappa([1_C])=[1_A]).$  Is there a unital monomorphism $\phi: C\to A$ such that
   $[\phi]=\kappa$? It is known (see \cite{Li-KT}) that if ${K}_i(C)$ is finitely generated then the answer is affirmative.
   It was proved in \cite{LnKT} that there exists a unital monomorphism
   $\phi: C\to A$ such that
   $[\phi]-\kappa=0$ in ${KL}(C,A).$  The problem remained  open whether
   one can choose $\phi$ so that $[\phi]=\kappa$ in ${KK}(C,A).$
   It also has been known that there are several significant consequences if
   such $\phi$ can be found.

   The second question is the following: Let $\phi: C\to A$ be a unital monomorphism and $\gamma\in \mathrm{Hom}(\Kone(C), \aff(\mathrm{T}(A))).$ Can we find
   a unital monomorphism $\psi: C\to B$ with $[\psi]=[\phi]$ in $KK(C, A)$ and the associated
   rotation map is
   $\gamma?$

   In this paper, we will give affirmative answers to both questions.  Among other
   consequences, we give the following:
   Let $A$ and $B$ be two unital simple \CA s which are inductive limits of type I \CA s with unique tracial states.
   Suppose that
   $$
   ({K}_0(A), {K}_0(A)_+, [1_A], {K}_1(A))\cong({K}_0(B), {K}_0(B)_+, [1_B], {K}_1(B)).
   $$
   Then $A$ and $B$ are ${\cal Z}$-stably isomorphic (see \ref{CM2} below).

     For the first question, as mentioned above, an affirmative answer was known for $C$ with finitely generated $K$-theory. Passing to inductive limits, the first author showed in \cite{LnKT} that any strictly positive element in $KL(C, A)$ can be lift for any unital AH-algebra $C$. However, since the $KK$-functor is not continuous with respect to inductive limits, it remained somewhat mysterious how to move from
     $KL(C,A)$ to $KK(C,A)$ until a
     clue was given by a paper of Kishimoto and Kumjian (\cite{KK}), where they studied
     simple $A\T$-algebras
     with real rank zero.  The important advantage of simple $A\T$-algebras is that
     their $K$-theory are torsion free. In this case a hidden  Bott-like map  was
     revealed. Kishimoto and Kumjian navigated this hurdle using the so-called
     Basic Homotopy Lemma in simple \CA s of real rank zero. We will take the idea of Kishimoto and Kumjian
     to the case that the domain algebras are no longer assumed the same as the targets. More generally,
     we will not assume that $C$ has real rank zero, nor will we assume it is simple.
     Furthermore, we will allow $C$ to have torsion in its $K$-theory. Therefore,
     for the general case,
     the Bott-like  maps involve the $K$-theory
     with coefficient and demand a much more general Basic Homotopy Lemma. For this
     we apply the recent established
     results in \cite{LnHomtp} and \cite{Lin-Asy}. It is also interesting to note that
     the second problem is closely related to the first one and its proof are also
     closely related.

     We will  consider the case that $C$ is a general unital AH-algebra (which may have arbitrary stable rank and other properties even in the case it is simple (see, for example, \cite{Vj}, \cite{Vj2} and \cite{EV})). Denote by $\mathrm{T}_{\mathtt{f}}(C)$ the convex set of all faithful tracial states of $C$. Let $\kappa\in {KK}_e(C,A)^{++}$ and let $\lambda: \mathrm{T}(A)\to \mathrm{T}_{\mathtt{f}}(C)$ be a continuous affine map. We say that $\lambda$ is compatible with $\kappa$ if $\lambda(\tau)(p)=\tau(\kappa([p]))$ for all projections $p$ in matrix algebras of $C$ and for all $\tau\in \mathrm{T}(A).$ We actually show that for any such
compatible pair $(\kappa, \lambda),$ there exists a unital monomorphism $\phi: C\to A$ such that $[\phi]=\kappa$ and
$\phi_\mathrm{T}=\lambda,$ where $\phi_\mathrm{T}:\mathrm{T}(A)\to \mathrm{T}_{\mathtt{f}}(C)$ is the induced continuous affine map induced by $\phi.$
It worth to point out that the information $\lambda$ is essential since there are examples of compact metric spaces
$X$, unital simple AF-algebras and $\kappa\in {K}_e(C,A)^{++}$ for which there is no unital
 monomorphism $h$ such that $[h]=\kappa$ (see \cite{Lnrange}). Furthermore, given a pair $([\phi], \phi_T)$ and
 $\gamma\in \mathrm{Hom}({K}_1(C), \aff(\mathrm{T}(A))$, there is a unital monomorphism $\psi: C\to A$ such that
 $([\psi],\psi_\mathrm{T})=([\phi],\phi_\mathrm{T})$ and a rotation map from ${K}_1(C)$ to $\aff(\mathrm{T}(A))$ associated
 with $\phi$ and $\psi$ is exactly $\gamma.$

      The paper is organized as follows: Preliminaries and notation are given in Section \ref{pre}. In Section \ref{kk}, it is shown that any element {\bf $\kappa\in{KK}_e(C, A)^{++}$} for which is   compatible
with a continuous affine map $\lambda: \mathrm{T}(A)\to \mathrm{T}_{\mathtt{f}}(C)$ can be
 represented by a monomorphism $\alpha$ if $C$ is a unital AH-algebra  and $A$ is a  unital simple C*-algebra with tracial rank zero. It is also shown that if $A=C$, then $\alpha$ can be chosen as an automorphism. Then, in Section \ref{rotation},  we prove that, for any monomorphism $\iota: C\to A$, one can realize any homomorphism $\psi$ from $\Kone(C)$ to $\aff(\mathrm{T}(A))$ as a rotation map without changing the $KK$-class of $\iota$, that is, there is a monomorphism $\alpha:C\to A$ such that $[\iota]=[\alpha]$ in $KK(C, A)$ and $\tilde{\eta}_{\iota, \alpha}=\psi$. Moreover, we also give a description of the asymptotical unitary equivalence class of the maps inducing the same $KK$-element. In Section \ref{cla}, we give an application of the results in the previous sections to the classification program. Combined with the work \cite{Winter-Z} of W. Winter and
 that of \cite{Lin-Asy}, it is shown that certain $\mathcal Z$-stable \CA s can be classified by their $K$-theory information.

\vskip 2mm
\noindent{\bf Acknowledgments.} Most of this research were conducted when both authors were visiting the Fields Institute in Fall 2007. We wish to thank the Fields Institute for the excellent working environment. The research of second named author is supported by an NSERC Postdoctoral Fellowship.

\section{Preliminaries and Notation}\label{pre}

\begin{NN}\label{NN1}
Let $A$ be a unital stably finite \CA. Denote by $\mathrm{T}(A)$ the simplex of tracial
states of $A$ and denote by $\textrm{Aff(T}(A))$ the space of all real
affine continuous functions on $\mathrm{T}(A).$ Suppose that $\tau\in \mathrm{T}(A)$ is a
tracial state. We will also use $\tau$ for the trace $\tau\otimes
\mathrm{Tr}$ on $\textrm{M}_k(A)=A\otimes \textrm{M}_k(\Comp)$ (for every integer $k\ge 1$), where
$\mathrm{Tr}$ is the standard trace on $\mathrm{M}_k(\Comp).$  A trace $\tau$ is faithful if $\tau(a)>0$ for any $a\in A_+\setminus\{0\}.$ Denote by $\mathrm{T}_{\mathtt{f}}(A)$ the convex subset of $\mathrm{T}(A)$ which consists of all faithful tracial states.

Denote by $\textrm{M}_{\infty}(A)$ the set $\displaystyle{\bigcup_{k=1}^{\infty}\textrm{M}_k(A)},$ where $\textrm{M}_k(A)$ is regarded as a C*-subalgebra of $\textrm{M}_{k+1}(A)$ by the embedding $$\textrm{M}_k(A)\ni a\mapsto \left(\begin{array}{cc}a&0\\0&0\end{array}\right)\in \textrm{M}_{k+1}(A)$$

Define the positive homomorphism $\rho_A: \Kzero(A)\to \textrm{Aff(T}(A))$
by $\rho_A([p])(\tau)=\tau(p)$ for any projection $p$ in
$\mathrm{M}_k(A).$ We also denote by $\mathrm{S}(A)$ the suspension of $A$, $\mathrm{U}(A)$ the unitary group of $A$, and $\mathrm{U}_0(A)$ the connected component of $\mathrm{U}(A)$ containing the identity.

Suppose that $C$ is another unital \CA\, and $\phi: C\to A$ is a unital *-homomorphism. Denote by $\phi_{\mathrm{T}}: \mathrm{T}(A)\to \mathrm{T}(C)$
the continuous affine map induced by $\phi,$ i.e., $\phi_\mathrm{T}(\tau)(c)=\tau\circ \phi(c)$ for all $c\in C$ and $\tau\in \mathrm{T}(A).$
\end{NN}

\begin{defn}\label{D-ad}
Let $A$ be a unital C*-algebra and let $B\subseteq A$ be a unital C*-subalgebra. For any $u\in\mathrm{U}(A)$, the *-homomorphism $\mathrm{Ad}(u)$ is defined by $$\mathrm{Ad}(u): B\ni b\mapsto u^*bu\in A.$$ Denote by $\overline{\mathrm{Inn}}(B, A)$ the closure of $\{\mathrm{Ad}(u);\ u\in\mathrm{U}(A)\}$ in $\mathrm{Hom}(B, A)$ with the pointwise convergence topology. Note that $\overline{\mathrm{Inn}}(A)\subseteq\overline{\mathrm{Inn}}(B, A)$.
\end{defn}

\begin{defn}\label{D-Maf}
Let $A$ and $B$ be two unital C*-algebras, and let $\psi$ and $\phi$ be two unital *-monomorphisms from $B$ to $A$. Then the mapping torus $M_{\psi, \phi}$ is the C*-algebra defined by $$M_{\psi, \phi}:=\{f\in\textrm{C}([0, 1], A);\ f(0)=\psi(b)\ \textrm{and}\ f(1)=\phi(b)\ \textrm{for some}\ b\in B\}.$$

If $B\subseteq A$ is a unital C*-subalgebra with $\iota$ the inclusion map, then, for any unital *-monomorphism $\alpha$ from $B$ to $A$, we also denote $M_{\iota, \alpha}$ by $M_\alpha.$

For any $\psi, \phi\in\mathrm{Hom}(B, A)$, denoting by $\pi_0$ the evaluation of $M_{\psi, \phi}$ at $0$, we have the short exact sequence
\begin{displaymath}
\xymatrix{
0\ar[r]&\mathrm{S}(A)\ar[r]^{\imath}&M_{\psi,\phi}\ar[r]^{\pi_0}&B\ar[r]&0,}
\end{displaymath}
and hence the six-term exact sequence
\begin{displaymath}
\xymatrix{
\Kzero(\mathrm{S}(A))\ar[r]^{\imath_0}&\Kzero(M_{\psi, \phi})\ar[r]^{[\pi_0]_0}&\Kzero(B)\ar[d]^{[\psi]_0-[\phi]_0}\\
\Kone(B)\ar[u]^{[\psi]_1-[\phi]_1}&\Kone(M_{\psi, \phi})\ar[l]^{[\pi_0]_1}&\Kone(\mathrm{S}(A)).\ar[l]^{\imath_1}
}
\end{displaymath}
If $[\psi]_*=[\phi]_*$, (in particular, if $B\subseteq A$ and $\alpha\in\overline{\mathrm{Inn}}(B, A)$), then the six-term exact sequence above breaks down to the following two extensions:
\begin{displaymath}
\eta_0({M}_{\psi, \phi}):\quad
\xymatrix{
0\ar[r]&\Kone(A)\ar[r]&\Kzero(M_{\psi, \phi})\ar[r]&\Kzero(B)\ar[r]&0,\\
}
\end{displaymath}
and
\begin{displaymath}
\eta_1({M}_{\psi, \phi}):\quad
\xymatrix{
0\ar[r]&\Kzero(A)\ar[r]&\Kone(M_{\psi, \phi})\ar[r]&\Kone(B)\ar[r]&0.\\
}
\end{displaymath}
\end{defn}
Moreover, if $[\psi]=[\phi]$ in $KL(B, A)$, the two extensions above are pure.

\begin{NN}
Suppose that, in addition,
\beq\label{Dr-2}
\tau\circ \phi=\tau\circ \psi\tforal \tau\in \textrm{T}(A).
\eneq
For any piecewise smooth path $u(t)\in M_{\psi, \phi}$, the integral
$$
R_{\phi, \psi}(u(t))(\tau)=\frac{1}{2\pi i}\int_0^1\tau(\dot{u}(t)u^*(t))\mathrm{d}t
$$
defines an affine function on $\mathrm{T}(A),$
and it depends only on the homotopy class of $u(t)$. Therefore, it induces a map, denoted by $R_{\psi, \phi}$, from $\Kone(M_\alpha)$ to $\mathrm{Aff}(\mathrm{T}(A))$, and we call it the rotation map. The map $R_{\psi, \phi}$ is in fact a homomorphism.

\end{NN}

\begin{NN}\label{loop}
If $p$ and $q$ be two  mutually orthogonal projections in $\textrm{M}_l(A)$ for some integer $l\ge 1$,
define  a unitary $u\in \textrm{U}({\widetilde{\textrm{M}_l(\textrm{S}(A))}})$
 by
$$
u(t)=(e^{2\pi it}p+(1-p))(e^{-2\pi it}q+(1-q))\,\,\,\text{for}\,\,\, t\in [0,1].
$$
One computes that
$$
\int_0^1 \tau( {d u(t)\over{dt}}u(t)^*)dt=\tau(p)-\tau(q)\tforal t\in [0,1].
$$

Note that if $v(t)\in {\widetilde{\textrm{M}_l(\textrm{S}A)}}$ is another piecewise smooth unitary which
is homotopic to $u(t),$ then, as mentioned above,
$$
\int_0^1\tau({d v(t)\over{dt}}) dt=\tau(p)-\tau(q).
$$
It follows that, for any two projections $p$ and $q$ in $\textrm{M}_l(A),$
$$
R_{\phi, \psi}(\imath_1([p]-[q]))(\tau)=\tau(p)-\tau(q) \tforal \tau\in \textrm{T}(A).
$$
In other words,
$$
R_{\phi,
\psi}(\imath_1([p]-[q]))={{\rho_{A}([p]-[q])}}.
$$
Thus one has, exactly as in 2.2 of \cite{KK},  the following:
\end{NN}

\begin{lem}\label{DrL} When 
${\rrm{ (\ref{Dr-2})}}$ holds, the following diagram commutes:
$$
\begin{array}{ccccc}
\Kzero(A) && \stackrel{\imath_1}{\longrightarrow} && \Kone(M_{\phi, \psi})\\
& \rho_A\searrow && \swarrow R_{\phi,\psi} \\
& & \mathrm{Aff(T}(A)) \\
\end{array}
$$

\end{lem}

\begin{defn}
Let $A$ and $C$ be two unital \CA s and let $\phi, \psi: C\to A$
be two unital homomorphisms. We say that $\phi$ and $\psi$ are
asymptotically unitarily equivalent if there exists a continuous
path of unitaries $\{u(t): t\in [0, \infty)\}$ such that
$$
\lim_{t\to\infty}{\rrm{Ad}}(u(t))\circ \phi(c)=\psi(c)\tforal c\in
C.
$$

\end{defn}

\begin{defn}\label{DKL}

Let $A$ be a  unital \CA\, and let $C$ be a  separable \CA\, which satisfies the Universal Coefficient Theorem.
By \cite{DL} of D{\u a}d{\u a}rlat and Loring,
\beq\label{N2-1}
{KL}(C,A)=\mathrm{Hom}_{\Lambda}(\underline{{K}}(C), \underline{{K}}(A)),
\eneq
where, for any \CA\, $B,$
$$
\underline{{K}}(B)=({K}_0(B)\oplus{K}_1(B))\oplus( \bigoplus_{n=2}^{\infty}({K}_0(B,\Z/n\Z)\oplus{K}_1(B,\Z/n\Z) ) ).
$$
We will identify the two objects in (\ref{N2-1}). Note that one may view ${KL}(C,A)$ as a quotient of ${KK}(C,A).$

Denote by ${KL}(C,A)^{++}$ the set of those ${\bar \kappa}\in \mathrm{Hom}_{\Lambda}(\underline{{K}}(C), \underline{{K}}(A))$ such that
$$
{\bar \kappa}(\Kzero(C)_+\setminus\{0\})\subseteq \Kzero^{+}(A)\setminus \{0\}.
$$
Denote by ${KL}_e(C, A)^{++}$ the set of those elements ${\bar \kappa}\in {KL}(C,A)^{++}$ such that ${\bar \kappa}([1_C])=[1_A].$

Suppose that both $A$ and $C$ are unital and $\mathrm{T}(C)\not=\varnothing$ and $\mathrm{T}(A)\not=\varnothing.$
Let $\lambda: \mathrm{T}(A)\to \mathrm{T}(C)$ be a continuous affine map. We say $\lambda$ is compatible with ${\bar \kappa}$ if
for any projection $p\in \mathrm{M}_{\infty}(C),$ one has that $\lambda(\tau)(p)=\tau(\bar{\kappa}([p])$ for all $\tau\in \mathrm{T}(A).$

Denote by ${KLT}(C,A)^{++}$ the set of those pairs $({\bar \kappa}, \lambda),$ where $\bar{\kappa}\in {KL}_e(C,A)^{++}$
and $\lambda: \mathrm{T}(A)\to \mathrm{T}_{\mathtt{f}}(C)$ is a continuous affine map which is compatible with $\bar{\kappa}.$
\end{defn}

\begin{defn}\label{KK+}
Denote by ${KK}(C,A)^{++}$ the set of those elements $\kappa\in {KK}(C,A)$ such that its image
${\bar \kappa}$ is in ${KL}(C,A)^{++}.$ Denote by ${KK}_e(C,A)^{++}$ the set of those $\kappa\in {KK}(C,A)^{++}$ for which
${\bar \kappa}\in {KL}_e(C,A)^{++}.$

Denote by ${KKT}(C,A)^{++}$ the set of pairs $(\kappa, \lambda)$ such that
$({\bar \kappa}, \lambda)\in {KLT}(C,A)^{++}.$

\end{defn}

\begin{NN}
Let $A$ and $B$ be two unital C*-algebras. Let $h: A\to B$ be a homomorphism and $v\in\mathrm{U}(B)$ such that $$[h(g),\ v]=0\quad \textrm{for any}\ g\in A.$$ We then have a homomorphism $\overline{h}: A\otimes\mathrm{C}(\mathbb T)\to B$ by $f\otimes g\mapsto h(f)g(v)$ for any $f\in A$ and $g\in\mathrm{C}(\mathbb T)$. The tensor product induces two injectve homomorphisms:
$$\beta^{(0)}:\Kzero(A)\to\Kone(A\otimes\mathrm{C}(\mathbb T)),$$
and
$$\beta^{(1)}:\Kone(A)\to\Kzero(A\otimes\mathrm{C}(\mathbb T)).$$
The second one is the usual Bott map. Note, in this way, one writes $${K}_i(A\otimes\mathrm{C}(\mathbb T))={K}_i(A)\oplus\beta^{(i-1)}({K}_{i-1}(A)).$$ We use $\widehat{\beta^{(i)}}:{K}_i(A\otimes\mathrm{C}(\mathbb T))\to\beta^{(i-1)}({K}_i(A))$ for the projection.

For each integer $k\geq2$, one also have the following injective homomorphisms:
$$\beta_k^{(i)}:{K}_i(A, \Int/k\Int)\to{K}_{i-1}(A\otimes\mathrm{C}(\mathbb T), \Int/k\Int),\quad i=0, 1.$$ Thus, we write
$${K}_i(A\otimes\mathrm{C}(\mathbb T), \Int/k\Int)={K}_i(A, \Int/k\Int)\oplus\beta^{(i-1)}({K}_{i-1}(A), \Int/k\Int).$$ Denote by $\widehat{\beta^{(i)}_k}:{K}_i(A\otimes\mathrm{C}(\mathbb T), \Int/k\Int)\to\beta^{(i-1)}({K}_i(A), \Int/k\Int)$ similar to that of $\widehat{\beta^{(i)}}$. If $x\in\underline{{K}}(A)$, we use $\boldsymbol{\beta}(x)$ for $\beta^{(i)}(x)$ if $x\in{K}_i(A)$ and for $\beta^{(i)}_k(x)$ if $x\in{K}_i(A, \Int/k\Int)$. Thus we have a map $\boldsymbol{\beta}: \underline{{K}}(A)\to\underline{{K}}(A\otimes\mathrm{C}(\mathbb T))$ as well as $\widehat{\boldsymbol{\beta}}:\underline{{K}}(A\otimes \mathrm{C}(\mathbb T))\to\boldsymbol{\beta}( \underline{{K}})$. Therefore, we may write $\underline{{K}}( A\otimes \mathrm{C}(\mathbb T))=\underline{{K}}(A)\oplus\boldsymbol{\beta}(\underline{{K}}(A))$.

On the other hand, $\overline{h}$ induces homomorphisms $$\overline{h}_{*i, k}: {K}_i(A\otimes\mathrm{C}(\mathbb T), \Int/k\Int)\to {K}_i(B, \Int/k\Int),$$ $k=0, 2, ... ,$ and $i=0, 1$.

We use $\mathrm{Bott}(h, v)$ for all homomorphisms $\overline{h}_{*i, k}\circ\beta_k^{(i)}$, and we use $\mathrm{bott}_1(h, v)$ for the homomorphism $\overline{h}_{1, 0}\circ\beta^{(1)}:\Kone(A)\to\Kzero(B)$, and $\mathrm{bott}_0(h, v)$ for the homomorphism $\overline{h}_{0, 0}\circ\beta^{(0)}:\Kzero(A)\to\Kone(B)$. We also use $\mathrm{bott}(u, v)$ for the Bott element when $[u, v]=0$.
\end{NN}

\begin{NN}
 Given a finite subset $\mathcal P\subset\underline{{K}}(A)$, there exists a finite subset $\mathcal F\subset A$ and $\delta_0>0$ such that $$\mathrm{Bott}(h, v)|_{\mathcal P}$$ is well defined if $$\norm{[h(a), v]}<\delta_0$$ for all $a\in\mathcal F$. See 2.11 of \cite{Lin-Asy} and 2.10 of \cite{LnHomtp} for more details.
\end{NN}

\begin{defn}\label{TA0}
A unital simple C*-algebra $A$ has tracial rank zero, denoted by $\mathrm{TR}(A)=0$, if for any finite subset $\mathcal F\subset A$, any $\ep>0$, and nonzero $a\in A^+$, there are nonzero projection $p\in A$ and finite dimensional C*-subalgebra $F$ with $1_F=p$, such that
\begin{enumerate}
\item $\norm{[x,\ p]}\leq\ep$ for any $x\in\mathcal F$,
\item for any $x\in\mathcal F$, there is $x'\in F$ such that $\norm{pxp-x'}\leq\ep$, and
\item $1-p$ is Murray-von Neumann equivalent to a projection in $\overline{aAa}$.
\end{enumerate}
\end{defn}

\begin{NN}
Finally,  we will write $a\approx_{\ep} b$ if $\|a-b\|<\ep.$
\end{NN}

\section{KK-lifting}\label{kk}

Let $A$ be a unital C*-algebra. Fix $0<\delta^{\mathfrak{p}}\leq\min\{\delta_1, \delta_2\}$ where $\delta_1$ and $\delta_2$ are the constant of $\delta$ of Lemma 9.6 and Lemma 9.7
of \cite{Lin-Asy} respectively. Note that $\delta^{\mathfrak{p}}$ is universal and $\delta^{\mathfrak{p}}<1/4$ (therefore, the Bott element $\mathrm{bott}(u, v)$ is well-defined for any unitaries $u$ and $v$ with $\norm{[u, v]}<\delta^{\mathfrak{p}}$). For some integer $l\geq 1$, let $z$ be a unitary in $\mathrm{M}_l(A)$, and let $U(t)$ be a path of uniaries $U(t)\in\mathrm{C}([0, 1], \mathrm{U}(\textrm{M}_l(A)))$ with $U(0)=1$ and $\norm{[U(1), z]}\leq\delta^{\mathfrak{p}}$, we have that
$$U(1)zU^*(1)z^*=\exp(i\omega)$$
where $\omega=(1/i)(\log(U(1)zU^*(1)z^*))\in A_{s.a}$.   Define the element $\p(U,z)$ as follows:
$$
{{\p}}(U,z)(t)=\begin{cases} U((8/7)t)zU^*((8/7)t)z^* &  \tforal t\in [0,7/8];\\
                               \exp(i8(1-t)\omega) &\tforal t\in [7/8,1]\end{cases}.
                               $$
Then, $\p(U, z)(t)$ defines a loop of unitaries in $A$, and this gives a well-defined element in $\Kone(\mathrm{S}(A))\cong\Kzero(A)$. 

\begin{rem}\label{REMinverse}
For a path of uniaries $U(t)\in\mathrm{C}([0, 1], \mathrm{U}(\mathrm{M}_l(A)))$ for some integer $l\geq 1$. Denote by
\begin{displaymath}
S(t):=\left(
\begin{array}{cc}
U(t) &\\
&U^*(t)
\end{array}
\right)
\quad
\mathrm{and}
\quad
Z:=\left(
\begin{array}{cc}
z &\\
&z
\end{array}
\right).
\end{displaymath}
If $U(0)=1$ and $\norm{[U(1), z]}\leq\delta^{\mathfrak{p}}$, then $S(0)=1$ and $\norm{[S(1), Z]}\leq\delta^{\mathfrak{p}}$, and we can consider the $K$-element $[\mathtt{p}(S, Z)]$ and we have $$[\mathtt{p}(S, Z)]=[\mathtt{p}(U, z)]+[\mathtt{p}(U^*, z)].$$

Denote by
\begin{displaymath}
R(t, s):=
\left(
\begin{array}{cc}
U(t) &\\
&1
\end{array}
\right)
R(s)\left(
\begin{array}{cc}
1 &\\
&U^*(t)
\end{array}
\right)R^*(s)
\end{displaymath}
where
\begin{displaymath}
R(s)=\left(
\begin{array}{cc}
\cos(\frac{\pi s}{2})& \sin(\frac{\pi s}{2} )\\
-\sin(\frac{\pi s}{2}) & \cos(\frac{\pi s}{2})
\end{array}
\right).
\end{displaymath}
Then one has that $R(t, 0)=S(t)$, $R(t, 1)=1$, and for any $s\in[0, 1]$, $R(0, s)=1$ and $\norm{[R(1, s), Z]}\leq2\delta^{\mathfrak{p}}$. Therefore, the path $\mathtt{p}(S, Z)(t)$ is homotopic to the identity in $\mathrm{M}_2(\widetilde{\mathrm{S}(A)})$ by a small perturbation of  $$W(t, s):=R(t, s)ZR^*(t, s)Z^*,$$ and hence
$$
[\mathtt{p}(U,z)]+[\mathtt{p}(U^*,z)]=[\mathtt{p}(S,Z)]=0\,\,\,{\rrm{in}}\,\,\,\Kone(\mathrm{S}(A)).
$$
\end{rem}

\begin{lem}\label{well-defn}
Let $U(t)$ and $V(t)$ be two continuous and piecewise smooth paths of unitaries. Let $z_1$ and $z_2$ be unitaries in $A$ with $\norm{[U(1), z_1]}<\delta^{\mathfrak{p}}$ and $\norm{[V(1), z_2]}<\delta^{\mathfrak{p}}$. If $[\mathtt{p}(U,z_1)]=[\mathtt{p}(V,z_2)]$, then one has that $\mathrm{bott}(U(1), z_1)=\mathrm{bott}(V(1), z_2)$.
\end{lem}
\begin{proof}
Consider the path of unitaries $$S(t)=\mathrm{diag}(U(t), V^*(t))$$ and the unitary $$Z:=\mathrm{diag}(z_1, z_2)$$ in $\mathrm{M}_2(A)$. By Remark \ref{REMinverse}, one has
$$
[\p(V^*,z_2)]+[\p(V,z_2)]=0,
$$
and hence
$$
[\p(S^*,Z)]=[\p(U^*,z_1)]+[\p(V,z_2)]=0.$$ Therefore, $\p(S^*,Z)$ is homotopic to the identity in a matrix algebra of $\widetilde{(\mathrm{S}(A))}$. Hence, passing to the matrix algebra if necessary, there is a continuous path of unitaries
$W(t, s)$
such that $W(t, 0)=1=W(t,1)$ for any $t\in[0,1]$ and $W(1, t)=\p(S^*,Z)(t)$ for any $t\in[0, 1]$, and $W(t, 1)=1$.

Note that $\p(S^*, Z)(t)=S^*({8\over{7}}t)ZS({8\over{7}}t)Z^*$ for $t\in [0, 7/8]$ and
$$
\|\p(S^*,Z)(7/8)-1\|<\delta^{\mathfrak{p}}.
$$
Thus, by Lemma 9.6 of \cite{Lin-Asy},
$$
0=\mathrm{bott}(S(1), Z)=\mathrm{bott}(U(1), z_1)-\mathrm{bott}(V(1), z_2),
$$
as desired.
\end{proof}

\begin{lem}\label{inj}
Let $U(t)$ and $V(t)$ be two continuous and  piecewise smooth pathes of unitaries. Let $z_1$ and $z_2$ be unitaries in $A$ with $\norm{[U(t), z_1]}<\delta^{\mathfrak{p}}$ and $\norm{[V(t), z_2]}<\delta^{\mathfrak{p}}$. If $\mathrm{bott}(U(1), z_1)=\mathrm{bott}(V(1), z_2)$, then one has that $[\p(U, z_1)]=[\p(V, z_2)]$.
\end{lem}

\begin{proof}
Consider the path of unitaries $$W(t):=\mathrm{diag}(U(t), V^*(t))$$ and unitary $$Z:=\mathrm{diag}(z_1, z_2)$$ in $\mathrm{M}_2(A)$. Then, one has that $\mathrm{bott}(W(1), Z)=\mathrm{bott}(U(1), z_1)-\mathrm{bott}(V(1), z_2)=0$. By Lemma 9.7 of \cite{Lin-Asy}, one has $$[\p(W^*,Z)]=0,$$ and hence
\begin{equation}\label{eqn001}
[\p(U^*,z_1)]+[\p(V,z_2)]=0.
\end{equation}

On the other hand, by Remark \ref{REMinverse},
$$
[\p(V^*,z_2)]+[\p(V,z_2)]=0.$$ Together with (\ref{eqn001}), one has
$$[\p(U, z_1)]=[\p(V, z_2)],$$ as desired.
\end{proof}

\begin{NN}\label{map-gamma}

Denote by $C(A)$ the subset of $\Kone(\mathrm{S}(A))$ consisting of $[\p(U,z)(t)]$ for a path of unitaries $U(t)\in\mathrm{C}([0, 1], \mathrm{M}_l(A))$ and a unitary $z\in \mathrm{M}_l(A)$ for some $l\geq 1$ with $\norm{[U(1), z]}<\delta^{\mathfrak{p}}$. It is easy to verify that $C(A)$ is a subgroup. It follows from Lemma \ref{well-defn} and Lemma \ref{inj} that there is an injective map $\Lambda: C(A)\to\Kzero(A)$ defined by $$\Lambda: [\p(U, z)]\mapsto\mathrm{bott}(U(1), z).$$ Moreover, the map $\Lambda$ is a homomorphism.

\end{NN}

\begin{lem}\label{trace-preserve}
Consider a path of unitaries $U(t)\in\mathrm{C}([0, 1], \mathrm{M}_\infty(A))$ with $U(0)=1$ and $\norm{[z, U(1)]}<\delta^{\mathfrak{p}}$. Then we have
$$
\tau([\p(U,z)])=\tau(\mathrm{bott}_1(U(1), z))
$$
for any $\tau\in\mathrm{T}(A)$ if one considers $[\p(U,z)]$ as an element in $\Kzero(A)\cong\Kone(\mathrm{S}(A))$. In other words, $\tau(\Lambda(h))=\tau(h)$ for any $h\in C(A)$.
\end{lem}
\begin{proof}
Denote by $\omega=U(1)zU^*(1)z^*$, and
note that $[\p(U,z)]\in \Kone(\mathrm{S}(A))=\Kzero(A).$ Denote by
$$r(t)=\exp(\log(\omega)(1-t)).$$
We then have that (see \ref{loop}), for any $\tau\in\mathrm{T}(A)$,
\begin{eqnarray*}
\tau(\p(U,z))&=& \frac{1}{2\pi i}\int_{0}^{7/8}
\tau((U((8/7)t)zU^*((8/7)t)z^*)'(zU((8/7)t)z^*U^*((8/7)t))dt\\
&&\hspace{0.5in}+\frac{1}{2\pi i}\int_{7/8}^1\tau(\dot{r}(8(t-7/8))r^*(8(t-7/8)))dt\\
&=& \frac{1}{2\pi i}\int_{0}^{7/8}
\tau((U((8/7)t)zU^*((8/7)t))'(U((8/7)t)z^*U^*((8/7)t))dt\\
&&\hspace{0.5in}+
\frac{1}{2\pi i}\int_{7/8}^1\tau(\dot{r}(8(t-7/8))r^*(8(t-7/8)))dt\\
&=&\frac{1}{2\pi i}\int_{7/8}^1\tau(\dot{r}(8(t-7/8))r^*(8(t-7/8)))dt\,\,\,\quad\textrm{(by Lemma 4.2 of \cite{Lin-Asy})}\\
&=&-\frac{1}{2\pi i}\tau(\log(w))\\
&=&\tau(\mathrm{bott}(U(1), z))\,\,\,\quad\textrm{(by Theorem 3.6 of \cite{Lin-Asy})},
\end{eqnarray*}
as desired.
\end{proof}

\begin{defn}\label{Pbott}
Let $A$ be a unital \CA\, with $\mathrm{T}(A)\not=\varnothing$. We say that $A$ has Property (B1) if the following holds: For any unitary $z\in \mathrm{U}(\textrm{M}_k(A))$ (for some integer $k\ge 1$) with ${\textrm{sp}}(z)=\T$, there is a non-decreasing function $1/4>\dt_z(t)>0$ on $[0,1]$ with $\dt_z(0)=0$ such that for any $x\in \Kzero(A)$ with $|\tau(x)|\le \dt_z(\ep)$ for all $\tau\in \mathrm{T}(A),$ there exists a unitary $u\in \mathrm{M}_k(A)$ such that
\beq\label{Pbott-1}
\|[u,\, z]\|<\min\{\ep, \frac{1}{4}\}\andeqn \text{bott}_1(u,\,z)=x.
\eneq

Let $C$ be a unital separable \CA. Let $1/4>\Delta_c(t,{\cal F}, {\cal P}_0, {\cal P}_1,h)>0$ be a function defined on $t\in [0,1],$ the family of all finite subsets ${\cal F}\subset C,$ the family of all finite subsets ${\cal P}_0\subset \Kzero(C),$ and family of all finite subsets ${\cal P}_1\subset \Kone(C),$  and the set of all unital monomorphisms  $h: C\to A.$   We say that $A$ has Property (B2) associated with $C$ and $\Delta_c$ if the following holds:

For any unital monomorphism $h: C\to A,$ any $\ep>0,$ any  finite subset ${\cal F}\subset C,$ any finite subset ${\cal P}_0\subset K_0(C),$ and any finite subset ${\cal P}_1\subset K_1(C),$ there are finitely generated subgroups $G_0\subset K_0(C)$ and $G_1\subset K_1(C)$ with $\mathcal G_0$ and $\mathcal G_1$ the sets of generators respectively and  ${\cal P}_0\subset G_0$ and ${\cal P}_1\subset G_1$ satisfying the following: for any homomorphisms $b_0: G_0\to K_1(A)$ and $b_1: G_1\to K_0(A)$ such that
\beq\label{Pbott-2}
\abs{\tau\circ b_1(g)}<\Delta_c(\ep, {\cal F}, {\cal P}_0, {\cal P}_1, h)
\eneq
for any $g\in {\cal G}_1$ and any $\tau\in\mathrm{T}(A)$, there exists a unitary $u\in \mathrm{U}(A)$ such that
\beq\label{Pbott-3}
\text{bott}_0(h, \, u)|_{{\cal P}_0}=b_0|_{{\cal P}_0},\,\,\, \text{bott}_1(h,\, u)|_{{\cal P}_1}=b_1|_{{\cal P}_1}\andeqn\\
\|[h(c),\, u]\|<\ep\tforal c\in {\cal F}.
\eneq

\end{defn}

\begin{rem}\label{B2}
Note that in the definition of Property (B2), $b_0$ and $b_1$ are defined on $G_0$ and $G_1,$ respectively, not
on the subgroups generated by ${\cal P}_0$ and ${\cal P}_1.$
One should also note that if $A$ has Property (B2) { associated with any C*-subalgebra $\textrm{C}(\mathbb{T})$}, then $A$ also has Property (B1).
\end{rem}

\begin{rem}\label{B2sym}
Let $A$ be a C*-algebra with Property (B2) associated with $C$ and $\Delta_c(t, \mathcal F, \mathcal P_0, \mathcal P_1, h)$. 
Then the function $\Delta'_c$ and the subgroups $G_0$ and $G_1$ can be chosen
so that they only depend on the unitary conjugate class of the unital embedding $h$.

Indeed, pick one representative $h_\lambda$ for each conjugate class of the unital embedding of $C$ to $A$, and define the function $$\Delta'_c(t, \mathcal F, \mathcal P_0, \mathcal P_1, h)=\Delta_c(t, \mathcal F, \mathcal P_0, \mathcal P_1, h_\lambda)$$ where $h_\lambda$ is in the unitary conjugate class of $h$. It is clear that $\Delta'_c$ only depends on the conjugate class of $h$. Let us show that $A$ has Property (B2) associated with $C$ and $\Delta'_c$ and $G_0$ and $G_1$
can be chosen so that they only depend on $h_\lambda.$

Fix a unital monomorphism $h: C\to A,$ any $\ep>0,$ a  finite subset ${\cal F}\subset C,$ a finite subset ${\cal P}_0\subset K_0(C),$ and a finite subset ${\cal P}_1\subset K_1(C)$. There is a unitary $w\in A$ such that $\textrm{Ad}(w)\circ h=h_\lambda$ for some $\lambda$.

Since $A$ has Property (B2) associated with $C$ and $\Delta_c$, there are finitely generated subgroups $G_0\subset K_0(C)$ and $G_1\subset K_1(C)$ with $\mathcal G_0$ and $\mathcal G_1$ the sets of generators respectively and  ${\cal P}_0\subset G_0$ and ${\cal P}_1\subset G_1$ satisfying the following: for any homomorphisms $b_0: G_0\to K_1(A)$ and $b_1: G_1\to K_0(A)$ such that
\beq\label{Pbott-4}
\abs{\tau\circ b_1(g)}<\Delta_c(\ep, {\cal F}, {\cal P}_0, {\cal P}_1, h_\lambda)
\eneq
for any $g\in {\cal G}_1$ and any $\tau\in\mathrm{T}(A)$, there exists a unitary $u\in \mathrm{U}(A)$ such that
\beq\label{Pbott-5}
\text{bott}_0(h_\lambda, \, u)|_{{\cal P}_0}=b_0|_{{\cal P}_0},\,\,\, \text{bott}_1(h_\lambda,\, u)|_{{\cal P}_1}=b_1|_{{\cal P}_1}\andeqn\\
\|[h_\lambda(c),\, u]\|<\ep\tforal c\in {\cal F}.
\eneq

Then,
\beq\label{Pbott-6}
&&\text{bott}_0(h,wuw^*)=\text{bott}_0(h_\lambda, \, u)|_{{\cal P}_0}=b_0|_{{\cal P}_0},\,\,\,
\text{bott}_1(h, wuw^*)=\text{bott}_1(h_\lambda,\, u)|_{{\cal P}_1}=b_1|_{{\cal P}_1}\\
&&\hspace{1in}\andeqn \|[h(c),\, wuw^*]\|<\ep\tforal c\in {\cal F}.
\eneq

In other words,  $A$ has Property (B2) associated with $\Delta'_c.$ Therefore, we can always assume that the function $\Delta_c$ and the subgroups $G_0$ and $G_1$ only depend on the conjugate classes of embeddings of $C$ into $A$.
\end{rem}

\begin{lem}\label{commutator}
Let $A$ be a unital C*-algebra which contains a
positive element $b$ with $\mathrm{sp}(b)=[0,1],$ and assume that $A$ has
Property ${\mathrm{(B1)}}$. There exists
$\delta>0$ such that for any $a\in\Kzero(A)$, if
$\abs{\tau(a)}\leq\delta$ for any $\tau\in\mathrm{T}(A)$, then one
has
$$a=[\p(U,z)]\in\Kone(\mathrm{S}(A))$$ for some unitary $z\in \mathrm{M}_k(A)$ (for some $k\ge 1$) and
some $U(t)\in\mathrm{C}([0, 1], \mathrm{U}(\mathrm{M}_k(A)))$ with $U(0)=1$
and $\norm{[U(1), z]}\leq\delta^{\mathfrak{p}}$. In other words, $a\in C(A)$.
\end{lem}
\begin{proof}

Since $A$ has Property (B1), for any $1/4>\ep>0$ and any unitary $z$ in a matrix of $A$ with $\mathrm{sp}(z)=\mathbb T$, there is a non-decreasing positive function $\delta_z(t)$ such that
if $a\in\Kzero(A)$ with $|\tau(a)|\leq \delta_z(\ep)$ for any trace $\tau$, then there is a unitary $u$ in the matrix of $A$ such that $\norm{[u, z]}<\ep$ and $\mathrm{bott}_1(u, z)=a$. 

For any $\ep$, there is $\delta_e(\ep)$ such that if $\abs{x-y}\leq\delta_e(\ep)$ with $x, y\in\{c\in\Comp;\ \abs{c-1}\leq\frac{1}{2}\}$, one has that $\abs{\log(x)-\log(y)}\leq\ep$. We also regard $\delta_e(\ep)$ as a positive function of $\ep$ with $\delta_e(0)=0$ and $\delta_e(\ep)>0$ if $\ep>0$.

For any natural number $k$ and any $\ep>0$, define
\begin{displaymath}
\Delta_z(\ep, k):=
\left\{
\begin{array}{ll}
\min\{\delta_z(\frac{1}{2}\Delta_z(\ep, k-1)),\frac{1}{2}\delta_e(\frac{1}{2}\Delta_z(\ep, k-1)), \delta^{\mathfrak{p}}, \ep\}&\textrm{if}\ k\geq 2,\\
\min\{\delta_z(\ep), \delta^{\mathfrak{p}}\}&\textrm{if}\ k=1.
\end{array}
\right.
\end{displaymath}
It is a positive function of $\ep$ and $k$ with $\Delta_z(\ep, k)>0$ if $\ep>0$.

Fix $0<\ep<1/12$. Then, there exists $m\in\mathbb{N}$ such that there there is a partition $0=t_0< t_1<\cdots<t_{m-1}<t_m=1$ such that for any $W(t)=(e^{2\pi it}p+(1-p))(e^{-2\pi it}q+(1-q))\in\mathrm{U}_\infty(\widetilde{\mathrm{S}(A)})$ where $p$ and $q$ are any
projections, one has that $\norm{W(t_{i-1})-W({t_i})}\leq\ep$ for each $1\leq i\leq m$. Fix this partition.

Let $z_0=\exp(i2\pi b).$ Since $\mathrm{sp}(b)=[0,1],$ $z_0$ is a unitary
in $A$ with $\mathrm{sp}(z_0)=\mathbb{T}$. Denote by
$$\delta=\frac{1}{2}\min\{\Delta_{z_0}(\ep, 2m), \Delta_{z_0}(\ep,
2m-1), ... , \Delta_{z_0}(\ep, 1)\}.$$

Let $a$ be an element in $\Kzero(A)$ with $\abs{\tau(a)}<\delta$
for any $\tau\in\mathrm{T}(A).$ Without loss of generality, we may
assume that $a=[p]-[q]$ for some projections $p,\, q\in
\mathrm{M}_k(A)$ for some
 integer $k\ge 1.$
  Note that
$$
\abs{\tau(p)-\tau(q)}< \dt.
$$
Put
 $W_0(t)=(e^{2\pi it}p+(1-p))(e^{-2\pi it}q+(1-q)).$
To simplify notation, by replacing $A$ by $\mathrm{M}_k(A),$ without loss of generality, we may assume that $p, q\in A$ and $W_0(t)\in A$ for each $t\in [0,1].$

%

It follows from Lemma 8 of \cite{LM} that, for some large $n\ge 1,$ there is a unitary $V_t\in M_{n+1}(A),$
for each $t\in [0,1],$
\beq\label{comm-add1}
\|V_t^*zV_t-W(t)z\|< \dt/2,
\eneq
where
$$
W(t)={\rrm{diag}}(W_0(t),\overbrace{1,1, ... , 1}^n),
$$
and
$$
z={\rrm {diag}}(z_0, \omega, \omega^2,...,\omega^n).
$$
and where $\omega=e^{2\pi i/n+1}.$

We assert that one can find unitaries $V_0, V_1, ... , V_m$ such that
$$\norm{V_izV^*_i-W({t_i})z}<\Delta_{z_0}(\ep, 2m+1-2i)\quad\textrm{for any}\ 0\leq i\leq m,$$ and $$\mathrm{bott}_1(V^*_{i+1}V_i, z)=0\quad\textrm{for any}\ 0\leq i\leq m-1.$$

Assume that the unitaries $V_0, V_1, ... , V_{i}$ satisfy the condition above. As indicated above, one can find a unitary $V_{i+1}$ such that $$\norm{V_{i+1}zV^*_{i+1}-W({t_{i+1}})z}<\min\{\frac{1}{2}\delta_e(\frac{1}{2}\Delta_{z_0}(\ep, 2m-2i)), \frac{1}{2}\Delta_{z_0}(\ep, 2m-2i-1)\},$$ and denote $$b_i:=\mathrm{bott}_1(V^*_{i+1}V_i, z).$$ By Theorem 3.6 of \cite{Lin-Asy}, one has that for any $\tau\in\mathrm{T}(A)$
\begin{eqnarray*}
\abs{\tau(b_i)}&=&\frac{1}{2\pi}\abs{\tau(\log(z^*V^*_{i+1}V_izV^*_{i}V_{i+1}))}\\
&=&\frac{1}{2\pi}\abs{\tau(\log(V_{i+1}z^*V^*_{i+1}V_izV^*_{i}))}.
\end{eqnarray*}

Note that
\begin{eqnarray*}
\norm{V_{i+1}z^*V^*_{i+1}V_izV^*_{i}- z^*W^*({t_{i+1}})W({t_{i}})z}&<&\frac{1}{2}\delta_e(\frac{1}{2}\Delta_{z_0}(\ep, 2m-2i ))+\Delta_{z_0}(\ep, 2m-2i+1)\\
&<&\delta_e(\frac{1}{2}\Delta_{z_0}(\ep, 2m-2i )).
\end{eqnarray*}
Therefore, we have $$\norm{\log(V_{i+1}z^*V^*_{i+1}V_izV^*_{i})-\log(z^*W^*({t_{i+1}})W({t_{i}})z)}<\frac{1}{2}\Delta_{z_0}(\ep, 2m-2i),$$ and hence
\begin{eqnarray*}
\abs{\tau(b_i)}&<&\frac{1}{2\pi}\abs{\tau(\log(W^*({t_{i+1}})W({t_{i}})))}+{\frac{1}{2}\Delta_{z_0}(\ep, 2m-2i)} \\
&=&\frac{1}{2\pi}(t_{i+1}-t_i)\abs{(\tau(p)-\tau(q))}+ {\frac{1}{2}\Delta_{z_0}(\ep, 2m-2i)}\\
&<&\Delta_{z_0}(\ep, 2m-2i)< \delta_{z_0}(\frac{1}{2}\Delta_{z_0}(\ep, 2m-2i-1)) .
\end{eqnarray*}
Since $A$ has Property (B1), there is a unitary $v'_{i+1}\in A$ such that
$$\norm{v'_{i+1}z_0(v'_{i+1})^*-z_0}<\frac{1}{2}\Delta_{z_0}(\ep, 2m-2i-1)\quad\mathrm{and}\quad\mathrm{bott}_1(v'_{i+1}, z_0)=b_i.$$
Set $$v_{i+1}=\mathrm{diag}(v_{i+1}', 1,...,1).$$ We then have
$$\norm{v_{i+1}zv_{i+1}^*-z}= \norm{v'_{i+1}z_0(v'_{i+1})^*-z_0} <\frac{1}{2}\Delta_{z_0}(\ep, 2m-2i-1)$$ and
$$\mathrm{bott}_1(v_{i+1}, z)=\mathrm{bott}_1(v'_{i+1}, z_0)=b_i.$$
Therefore
\begin{eqnarray*}
\norm{(V_{i+1}v_{i+1})z(v^*_{i+1}V^*_{i+1})-W(t_{i+1})z}&<&\frac{1}{2}\Delta_{z_0}(\ep, 2m-2i-1)+\norm{V_{i+1}zV^*_{i+1}-W(t_{i+1})z}\\
&<&\Delta_{z_0}(\ep, 2m-2i-1)
\end{eqnarray*}
and $$\mathrm{bott}_1((v^*_{i+1}V^*_{i+1})V_{i}, z)= \mathrm{bott}_1(v^*_{i+1}, z)+\mathrm{bott}_1(V^*_{i+1}V_{i}, z)=0.$$
By replacing $V_{i+1}$ by $V_{i+1}v_{i+1}$, we proved the assertion.

Then, for each $V_i$ and $V_{i+1}$, there is a path of unitary $V^{(i)}(t)$ such that
$$
 \|[V^{(i)}(t), z]\|<\ep,\ V^{(i)}(0)=1,\
V^{(i)}(1)=V_{i+1}^*V_i.
$$
By setting $U^{(i)}(t)=V_{i+1}V^{(i)}({1-t})$, we have that $$U^{(i)}(0)=V_i,\ U^{(i)}(1)=V_{i+1},$$ and for each $t$,
\begin{eqnarray*}
\norm{U^{(i)}(t)z(U^{(i)}(t))^*-W(t)z}&=&\norm{V_{i+1}V^{(i)}({1-t})z(V^{(i)}({1-t}))^*V_{i+1}^*-W(t)z}\\
&<&\ep+ \norm{V_{i+1}zV_{i+1}^*-W(t)z}\\
&\leq&2\ep\leq\frac{1}{4}.
\end{eqnarray*}
By connecting all $U^{(i)}(t)$, we get a path of unitary $U(t)$ such that for any $t$, $$\norm{U(t)zU^*(t)-W(t)z}<\frac{1}{4}.$$ In particular, $$a=[W(t)]_1=[\p(U,z)].$$ Moreover, $$\norm{[U(1), z]}=\norm{[V_m, z]}<\delta^{\mathfrak{p}},$$ as desired.
\end{proof}



\begin{cor}\label{Ccomm}
Let $A$ be a unital separable simple \CA\ with $\mathrm{TR}(A)=0.$ Then
there exists $\delta>0$ such that for any element $a\in\Kzero(A)$, if
$\abs{\tau(a)}\leq\delta$ for any $\tau\in\mathrm{T}(A)$, then one
has
$$a=[\p(U,z)]\in\Kone(\mathrm{S}(A))$$ for some unitary $z\in \mathrm{M}_k(A)$ and
some $U(t)\in\mathrm{C}([0, 1], \mathrm{U}(\mathrm{M}_k(A)))$ (for some $k\ge
1$) with $U(0)=1$ and $\norm{[U(1), z]}<\delta^{\mathfrak{p}}$.
\end{cor}
\begin{proof}
By Lemma 5.2 of \cite{Lin-Asy}, $A$ has Property (B1). Then, the statement follows from Lemma \ref{commutator}.
\end{proof}

\begin{cor}\label{Ccomm1}
Let $A$ be a unital separable simple \CA\ with $\mathrm{TR}(A)=0.$ Then $C(A)=\Kzero(A)$.
\end{cor}
\begin{proof}
Denote by $\delta$ the constant of Corollary \ref{Ccomm}. For any $a\in\Kzero(A)$, since $\Kzero(A)$ is tracially approximately divisible (see Definition \ref{tadivisible}), one has $$a=a_1+a_2+\cdots+a_n$$ with $\abs{\tau(a_i)}<\delta$ for each $1\leq i\leq n$. Therefore, $a_i\in C(A)$. Since $C(A)$ is a group, one has that $a\in C(A)$, as desired.
\end{proof}

\begin{defn}\label{DR}
 For any unitary $u$ in a C*-algebra $A$, denote by $R(u, t)$ the
unitary path in $\mathrm{M}_2(A)$ defined by
\begin{displaymath}
R(u, t):=\left(
\begin{array}{cc}
u &0\\
0 &1
\end{array}
\right)
\left(
\begin{array}{cc}
\cos(\frac{\pi t}{2})& \sin(\frac{\pi t}{2} )\\
-\sin(\frac{\pi t}{2}) & \cos(\frac{\pi t}{2})
\end{array}
\right)
\left(
\begin{array}{cc}
u^*& 0\\
0 & 1
\end{array}
\right)
\left(
\begin{array}{cc}
\cos(\frac{\pi t}{2})& -\sin(\frac{\pi t}{2} )\\
\sin(\frac{\pi t}{2}) & \cos(\frac{\pi t}{2})
\end{array}
\right),\,\,\,\,\,\,t\in [0,1].
\end{displaymath}
Note that $$R(u,
0)=\left(\begin{array}{cc}1&0\\0&1\end{array}\right)\quad\mathrm{and}\quad
R(u, 1)=\left(\begin{array}{cc}u&0\\0&u^*\end{array}\right) $$
\end{defn}

\begin{lem}\label{decomp}
Let $A$ be a unital C*-algebra. Let $B\subseteq A$ be a unital separable C*-subalgebra. Write $$\Kzero(B)_+=\{k_1, k_2,  ... , k_n, ... \}$$ and $$\Kone(B)=\{h_1, h_2,  ... , h_n,  ... \},$$ and denote by  $K_n=<k_1, ... ,k_n>$ the subgroup generated by $\{k_1, ... ,k_n\}$, and by $H_n=<h_1, ... ,h_n>$ the subgroup generated by
$\{h_1, ... ,h_n\}.$  Let $\{\F_i\}$ be an increasing family of finite subsets whose union is dense in $B$. Assume that for each $i$, there exist a projection $p_i\in\mathrm{M}_{r_i}(\F_i)$ and a unitary $z_i\in\mathrm{M}_{r_i}(\F_i)$ with $[p_i]_0=k_i$ and $[z_i]_1=h_i$. Let $\{u_n\}$ be a sequence of unitaries such that
$$
\norm{[u_{n+1}^{(r_n)},\, a]}\leq\frac{\delta^{\mathfrak{p}}}{2^{n+1}}
$$
for any $a\in\mathrm{M}_{r_n}(w_n^*\F_nw_n)$, where $w_n=u_1\cdots u_n$ and
$u_k^{(r_j)}=\rrm{diag}(\overbrace{u_k,u_k,...,u_k}^{r_j}).$
 Then $$\alpha=\lim_{n\to\infty}\mathrm{Ad}(w_n)$$ defines a monomorphism from $B$ to $A$, and the extensions $\eta_0(M_\alpha)$ and $\eta_1(M_\alpha)$ are determined by the inductive limits

\vskip 3mm
\xy
(0,0)*{~}="0";
(20,15)*{0}="1";
(40,15)*{K_1(A)}="2";
(70,15)*{K_1(A)}="3";
(78,15)*{\oplus}="4";
(86,15)*{K_{n+1}}="5";
(110,15)*{K_{n+1}}="6";
(130, 15)*{0}="7";
(20,0)*{0}="1l";
(40,0)*{K_1(A)}="2l";
(70,0)*{K_1(A)}="3l";
(78,0)*{\oplus}="4l";
(86,0)*{K_{n}}="5l";
(110,0)*{K_{n}}="6l";
(130, 0)*{0}="7l";
{\ar "1";"2"};
{\ar "2";"3"};
{\ar "5";"6"};
{\ar "6";"7"};
{\ar "1l";"2l"};
{\ar "2l";"3l"};
{\ar "5l";"6l"};
{\ar "6l";"7l"};
{\ar@{=} "2l";"2"};
{\ar@{=} "3l";"3"};
{\ar_{\iota_{n, n+1}} "5l";"5"};
{\ar_{\iota_{n, n+1}} "6l";"6"};
{\ar_{\gamma^{0}_n} "5l";"3"};
\endxy
\vskip 3mm

\noindent and

\vskip 3mm
\xy
(0,0)*{~}="0";
(20,15)*{0}="1";
(40,15)*{K_0(A)}="2";
(70,15)*{K_0(A)}="3";
(78,15)*{\oplus}="4";
(86,15)*{H_{n+1}}="5";
(110,15)*{H_{n+1}}="6";
(130, 15)*{0}="7";
(20,0)*{0}="1l";
(40,0)*{K_0(A)}="2l";
(70,0)*{K_0(A)}="3l";
(78,0)*{\oplus}="4l";
(86,0)*{H_{n}}="5l";
(110,0)*{H_{n}}="6l";
(130, 0)*{0}="7l";
{\ar "1";"2"};
{\ar "2";"3"};
{\ar "5";"6"};
{\ar "6";"7"};
{\ar "1l";"2l"};
{\ar "2l";"3l"};
{\ar "5l";"6l"};
{\ar "6l";"7l"};
{\ar@{=} "2l";"2"};
{\ar@{=} "3l";"3"};
{\ar_{\iota_{n, n+1}} "5l";"5"};
{\ar_{\iota_{n, n+1}} "6l";"6"};
{\ar_{\gamma^{1}_n} "5l";"3"};
\endxy

\vskip 3mm

\noindent respectively, where $\gamma_n^0:K_n\to \Kone(A)$ is defined by $k_i\mapsto[(w_n^*p_iw_n^*)u_{n+1}(w_n^*p_iw_n)+(1-w_n^*p_iw_n)]$ and $\gamma_n^1:H_n\to \Kzero(A)$ is defined by $h_i\mapsto[\p(R^*(u_{n+1}, t), w_n^*z_iw_n)]$.
\end{lem}

\begin{proof}
It is clear that we may assume that $r_n\le r_{n+1},$ $n=1,2,...$

 Note that for any $a\in\F_n$,
\begin{eqnarray*}
\norm{\mathrm{Ad}(w_{n+k})(a)-\mathrm{Ad}(w_n)(a)}&\leq&\sum_{i=1}^k\norm{\mathrm{Ad}(u_1u_2\cdots u_{n+i})(a)-\mathrm{Ad}(u_1u_2\cdots u_{n+i-1})(a)}\\
&\leq&\sum_{i=1}^k\frac{1}{2^{n+i}}\leq\frac{1}{2^n}.
\end{eqnarray*}
Therefore $\lim\mathrm{Ad}(w_n)$ exists if $n\to\infty$. Denote by $\alpha_n=\mathrm{Ad}(w_n)$, and denote its limit by $\alpha$. Note that $\alpha$ is a monomorphism.
Moreover,
\beq\label{decom-add1}
[\alpha]=[\imath]\,\,\,\text{in}\,\,\,{KL}(B, A),
\eneq
where $\imath: B\to A$ is the embedding.
It follows that the six-term exact sequence in \ref{D-Maf} splits.

To simplify notation, without loss of generality, in what follows, we may replace $\alpha\otimes \mathrm {id}_{\mathrm{M}_{r_i}}$ by $\alpha,$ $\alpha_n\otimes \mathrm{id}_{\mathrm{M}_{r_i}}$ by $\alpha_n,$ $u_n^{(r_n)}$ by $u_n$, and $w_n^{(r_n)}$ by
$w_n$ respectively (and write $p_i\in A$).

Consider $\Kzero(M_\alpha(A))$ first. Fix $n$, and note that $$\norm{\alpha_n(p_i)-\alpha(p_i)}\leq\frac{1}{4},\quad\textrm{for any}\ 1\leq i\leq n.$$ Therefore, for each $i$, there is a unitary $v_i$ with $\norm{v_i-1}\leq\frac{1}{2}$ such that $\alpha(p_i)=v_i^*\alpha_n(p_i)v_i$. In particular, there is a path $r^{(n)}_i(t)$ of projection with $r^{(n)}_i(0)=\alpha_n(p_i)$, $r^{(n)}_i(1)=\alpha(p_i)$, and $\norm{r^{(n)}_i(t)-\alpha(p_i)}\leq\frac{1}{2}$. Then there is a homomorphism $\psi_{n}^{(0)}: \Kone(A)\oplus K_n\to\Kzero(M_{\alpha})$ defined by $$\psi_{n}^{(0)}: ([Q(t)], [p_i])\mapsto [Q(t)]+[P^{(n)}_i(t)],$$ where $P^{(n)}_i(t)$ is the path
\begin{displaymath}
P^{(n)}_i(t)=\left\{
\begin{array}{ll}
R^*(w_n, 2t)p_iR(w_n, 2t)& 0\leq t\leq \frac{1}{2}\\
r^{(n)}_i(2t-1)& \frac{1}{2}\leq t\leq 1,
\end{array}
\right.
\end{displaymath}
and $[Q(t)]\in \Kzero(\mathrm{S}(A))$.
Since $[\pi]_{0}\circ \psi_n^{(0)}={\mathrm{id}}_{K_n}, $ the map $\psi_n^{(0)}$ is injective.

We then have
\beq\label{decomp-add2}
\psi_{n+1}^{(0)}([Q(t)], [p_i])-\psi_{n}^{(0)}([Q(t)], [p_i])=
[P^{(n+1)}_i(t)]-[P^{(n)}_i(t)]\in\Kzero(\mathrm{S}(A)).
\eneq
Moreover, it is easy to see that
$$[P^{(n+1)}_i(t)]-[P^{(n)}_i(t)]=[R^*(u_{n+1},
t)w^*_np_iw_nR(u_{n+1},
t)]-[w^*_np_iw_n]\in\Kzero(\mathrm{S}(A)).$$

Define homomorphism $\psi_{n, n+1}^{(0)}:\Kone(A)\oplus K_n\to\Kone(A)\oplus K_{n+1}$
by
$$
\psi_{n, n+1}^{(0)}: (Q(t), [p_i(t)])\mapsto
([Q(t)]+([P^{(n)}_i(t)]-[P^{(n+1)}_i(t)]), [p_i(t)]).
$$
By the construction we have that $\psi_{n+1}^{(0)}\circ\psi_{n,
n+1}^{(0)}=\psi_n^{(0)}.$ Thus, we obtain a homomorphism
$$
\psi^{(0)}: \varinjlim(\Kzero(M_{\alpha_n}), \psi_{n, n+1})\to
\Kzero(M_\alpha(A)).
$$

Let us show that $\psi^{(0)}$ is surjective. For each projection $p(t)\in M_\alpha$,
we can assume that $[p(0)]_0=k_i=[p_i]_0\in K_n$ for some $i\leq n$, and if denote by
$v$ the partial isometry with $v^*p_iv=p(0)$, then
$\norm{w_n^*vw_n-v}\leq \frac{1}{2}$. Then there is a path of partial isometries
$v(t)\in M_\alpha$ such that $v(0)=v$ and $v(1)=\alpha(v)$. Then
$$h:=[p(t)]-[v^*(t)P_i^{(n)}(t)v(t)]$$ is an element in $\Kzero(\mathrm{S}(A))$ and $[p(t)]_0=\psi^{(0)}(h, k_i)$. Therefore $\psi^{(0)}$ is surjective. The injectivity of $\psi^{(0)}$ follows from the injectivity of each $\psi^{(0)}_n$.  Thus, $\psi^{(0)}$ is an isomorphism.

Let us show that $\psi_{n, n+1}^{(0)}$ has the desired form.

Consider the invertible element
$c=(w^*_np_iw_n)u_{n+1}(w^*_np_iw_n)+(1-(w^*_np_iw_n))\in A$
(which is close to a unitary). Let us calculate the corresponding
element in $\mathrm{S}(A)$. Consider the path
$$
Z(t)=\left(
\begin{array}{cc}
w^*_np_iw_n&0\\
0&w^*_np_iw_n
\end{array}
\right)
R(u_{n+1}, t)
\left(
\begin{array}{cc}
w^*_np_iw_n&0\\
0&w^*_np_iw_n
\end{array}
\right)+
\left(
\begin{array}{cc}
1-w_n^*p_iw_n&0\\
0&1-w_n^*p_iw_n
\end{array}
\right).
$$
Note that $Z(0)=\mathrm{diag}(1, 1)$ and
$Z(1)=\mathrm{diag}(c, c^*)$ and $Z(t)$ is invertible for any $t$
with $\norm{Z^*(t)Z(t)-1}\leq\frac{1}{2^n}$. Let
$$
e(t)=(Z^*(t)Z(t))^{-\frac{1}{2}}Z^*(t)\left(\begin{array}{cc}1
&0\\ 0& 0\end{array}\right)Z(t)(Z^*(t)Z(t))^{-\frac{1}{2}}.
$$
Then
the element in $\Kzero(\mathrm{S}(A))$ which corresponds to $c$ is
$[e]-[\left(\begin{array}{cc}1 &0\\ 0& 0\end{array}\right)]$.
However, since
$$\norm{\left[\left(
\begin{array}{cc}
w^*_np_iw_n&0\\
0&w^*_np_iw_n
\end{array}
\right),
R(u_{n+1}, t)\right]}\leq2\norm{[w^*_np_iw_n, u_{n+1}]}\leq\frac{1}{2^n},
$$
a direct calculation shows
\begin{eqnarray*}
e(t)&\approx_{2/(2^n-1)} &\left(
\begin{array}{cc}
w^*_np_iw_n&0\\
0&w^*_np_iw_n
\end{array}
\right)
R^*(u_{n+1}, t)
\left(
\begin{array}{cc}
w^*_np_iw_n&0\\
0&0
\end{array}
\right)
R(u_{n+1}, t)
\left(
\begin{array}{cc}
w^*_np_iw_n&0\\
0&w^*_np_iw_n
\end{array}
\right)\\
&&+\left(
\begin{array}{cc}
1-w^*_np_iw_n&0\\
0&0
\end{array}
\right)\\
&\approx_{2/2^n}&
R^*(u_{n+1}, t)\left(
\begin{array}{cc}
w^*_np_iw_n&0\\
0&w^*_np_iw_n
\end{array}
\right)
\left(
\begin{array}{cc}
w^*_np_iw_n&0\\
0&0
\end{array}
\right)
\left(
\begin{array}{cc}
w^*_np_iw_n&0\\
0&w^*_np_iw_n
\end{array}
\right)
R(u_{n+1}, t)
\\
&&+\left(
\begin{array}{cc}
1-w^*_np_iw_n&0\\
0&0
\end{array}
\right)\\
&=&
R^*(u_{n+1}, t)
\left(
\begin{array}{cc}
w^*_np_iw_n&0\\
0&0
\end{array}
\right)
R(u_{n+1}, t)
+\left(
\begin{array}{cc}
1-w^*_np_iw_n&0\\
0&0
\end{array}
\right),
\end{eqnarray*}
and therefore, for $n>2$,
\begin{eqnarray*}
[e]-[\left(\begin{array}{cc}1 &0\\ 0& 0\end{array}\right)]
&=&[
R^*(u_{n+1}, t)
\left(
\begin{array}{cc}
w^*_np_iw_n&0\\
0&0
\end{array}
\right)
R(u_{n+1}, t)]
-[\left(
\begin{array}{cc}
w^*_np_iw_n&0\\
0&0
\end{array}
\right)]\\
&=&[R^*(u_{n+1}, t)w^*_np_iw_nR(u_{n+1}, t)]-[w^*_np_iw_n]\\
&=&[P^{(n+1)}_i(t)]-[P^{(n)}_i(t)].
\end{eqnarray*}
Therefore, the corresponding element of $[P^{(n+1)}_i(t)]-[P^{(n)}_i(t)]$ in $\Kone(A)$ is $$(w^*_np_iw_n)u_{n+1}(w^*_np_iw_n)+(1-(w^*_np_iw_n)),$$ and hence the map $\psi_{n, n+1}^{(0)}$ has the desired form.

Now, let us consider $\Kone(M_\alpha)$. For each $n$, consider the unitaries $\{z_i;\ 1\leq i\leq n\}$. Note that $$\norm{\alpha(z_i)-\alpha_n(z_i)}\leq\frac{1}{2^n}$$ for each $i$. Then,
there are paths of unitaries $s_i^{(n)}$ such that $s_i^{(n)}(0)=\alpha_n(v_i)$, $s_i^{(n)}(1)=\alpha(v_i)$, and $$\norm{s_i^{(n)}(t)-\alpha(v_i)}\leq\frac{1}{2^{n-1}}$$ for each $t\in[0,1]$. Define a homomorphism $\psi_n^{(1)}: \Kzero(A)\bigoplus H_n\to\Kone(M_\alpha(A))$ by $$\psi_n^{(1)}: ([S(t)], z_i)\mapsto [S(t)]+[V_i^{(n)}(t)]$$ where $V_i^{(n)}$ is the path
\begin{displaymath}
V^{(n)}_i(t)=\left\{
\begin{array}{ll}
R^*(w_n, 2t)z_iR(w_n, 2t)& 0\leq t\leq \frac{1}{2}\\
s^{(n)}_i(2t-1)& \frac{1}{2}\leq t\leq 1,
\end{array}
\right.
\end{displaymath}
and $S(t)$ is a unitary in
$\mathrm{U}_\infty(\widetilde{\mathrm{S}(A)})$.

We then have
\begin{eqnarray*}
\psi_{n+1}([S(t)],[v_i])-\psi_{n}([S(t)],[v_i])&=&[V^{(n+1)}_i(t)]-[V^{(n)}_i(t)]\\
&=&[\p(R^*(u_{n+1}, t), w^*_nv_iw_n)]\in\Kone(\mathrm{S}(A)).
\end{eqnarray*}
Define homomorphism $\psi_{n, n+1}^{(1)}:\Kzero(A)\oplus H_n\to\Kzero(A)\oplus H_{n+1}$ by $$\psi_{n, n+1}^{(0)}: (S(t), [v_i])\mapsto ([S(t)]+([V^{(n)}_i(t)]-[V^{(n+1)}_i(t)]), [v_i]).$$ By the construction we have that $\psi_{n+1}^{(1)}\circ\psi_{n, n+1}^{(1)}=\psi_n^{(1)}.$ Thus, there is a homomorphism $$\psi^{(1)}: \varinjlim(\Kzero(M_{\alpha_n}), \psi_{n, n+1})\to\Kzero(M_\alpha).$$ By the same argument as that for $\Kzero(M)_\alpha$, the map $\psi^{(1)}$ is an isomorphism, and moreover, it has the desired form.
\end{proof}

\begin{NN}\label{psp}

Let
\begin{displaymath}
\xymatrix{
0\ar[r]&G\ar[r]&E\ar[r]^\pi&H\ar[r]&0}
\end{displaymath}
be a short exact sequence of abelian groups with $H$ countable.  We assume that the extension is pure, i.e.,
for any finitely generated subgroup $H'\subseteq H$, there is a homomorphism $\theta': H'\to E$ such that $\pi\circ \theta'=\text{id}_{H'}.$

Write $H=\{h_1, h_2,  ... \}$. Consider $H_n=<h_1,  ... , h_n>$,
the subgroup generated by $\{h_1,  ... , h_n\}$. Since $E$ is a
pure extension, there is a map $\theta_n: H_n\to E$ such that
$\pi\circ\theta_n=\mathrm{id}_{H_n}$. Let us call this map
$\theta_n$ a partial splitting map. Denote by $\iota_n$ the
inclusion map from $H_n$ to $H_{n+1}$, and set the map $\gamma_n:
H_n\to E$ by
$$\gamma_n=\theta_{n+1}\circ\iota_n-\theta_n.$$ It is clear that
$\gamma_n(H_n)\subseteq G\subseteq E$. Therefore, let us regard
$\gamma_n$ as a map from $H_n$ to $G$.

\end{NN}

\begin{defn}\label{tadivisible}
A partially ordered group $G$ with an order unit is
tracially approximately divisible if for any $a\in G$, any
$\ep>0$, and any natural number $n$, there exist $b\in G$ such
that $\abs{\tau(a-nb)}\leq\ep$ for any state $\tau$ of $G$.
\end{defn}

\begin{lem}\label{smalltrace}
With the setting as \ref{psp}, if, in addition, $G$ is a partially ordered group
with an order unit which is
tracially approximately divisible, then, for any $\ep>0$
and any given partial splitting map $\theta_n$ with $\pi\circ
\theta_n=\mathrm{id}_{H_n},$ one can choose a partial splitting map
$\theta_{n+1}$ such that $\abs{\tau(\gamma_n(h_i))}<\ep$ for any
$1\leq i\leq n$ and any state $\tau$ of $G$.
\end{lem}
\begin{proof}
Pick any partial splitting map $\theta_{n+1}': H_{n+1}\to E_1$, and consider the map $\gamma'_n: H_n\to G$ defined by $$\gamma'_{n}=\theta'_{n+1}\circ\iota_n-\theta_n.$$ Consider the $\Ratn$-linear map
$$
\gamma_n'\otimes\mathrm{id}: H_{n}\bigotimes\Ratn\to G \bigotimes\Ratn.
$$
Then it has an extenstion to $H_{n+1}\bigotimes\Ratn$. That is, there is an linear map $\psi: H_{n+1}\bigotimes\Ratn\to G \bigotimes\Ratn$ such that
\begin{displaymath}
\xymatrix{
G \bigotimes\Ratn&H_{n+1}\bigotimes\Ratn \ar[l]_{\psi}\\
&H_n\bigotimes\Ratn\ar[ul]^{\gamma'_n\otimes\mathrm{id}}\ar[u]_{\iota_n\otimes\mathrm{id}}
}
\end{displaymath}
commutes.

Since $H_{n+1}$ is finitely generated, $H_{n+1}\cong\Int^{k_{n+1}}\bigoplus T_{n+1}$ for some finite abelian group $T_{n+1}$. Denote by $\{e_1, e_2,  ... , e_{k_{n+1}}\}$ the standard generators for the torsion free part of $H_{n+1}$. Then, for each $1\leq i\leq k_{n+1}$, we have $$\psi(e_i)=\sum_{j=1}^{l_i}r_j^{(i)}g_j^{(i)}$$ for some $r_j^{(i)}\in\Ratn$ and $g_j^{(i)}\in G$.

Write $(\iota_n(h_s))_{\mathrm{free}}=\sum_{i}m_i^{(s)}e_i$, and denote by $m=\max\{m_i^{(s)}; 1\leq s\leq n, 1\leq i\leq k_{n+1}\}$. Since $G$ is tracially approximately divisible, for each $e_i$, one can find $p_i\in G$ such that $$\tau(\psi(e_i)-p_i)<\frac{\ep}{k_{n+1}m}$$ for any state $\tau$ of $G$.

Define the map $\phi: H_{n+1}\to\Kzero(A)\subset E_1$ by sending $e_i$ to $p_i$ and $T_{n+1}$ to $\{0\}$, and let us consider the map $$\theta_{n+1}:=\theta'_{n+1}-\phi.$$
Then, the map $\theta_{n+1}$ satisfies the lemma. Indeed, it is clear that $\pi\circ\theta_{n+1}=\mathrm{id}_{H_{n+1}}$. Moreover, for any $h_s$ and any $\tau$, we have
\begin{eqnarray*}
\abs{\tau(\theta_{n+1}\circ\iota_n(h_s)-\theta_n(h_s))}&=&\abs{\tau(\theta'_{n+1}\circ\iota_n(h_s)-\phi\circ\iota_n(h_s)-\theta_n(h_s))}\\
&=& \abs{\tau(\gamma_n'(h_s)-\phi\circ\iota_n(h_s))}\\
&=&\abs{\tau(\psi\circ\iota_n(h_s)-\phi\circ\iota_n(h_s))}\\
&=&\abs{\sum_{i}m_i^{(s)}\tau(\psi(e_i)-p_i)}\\
&<&\ep,
\end{eqnarray*}
as desired.
\end{proof}

\begin{thm}\label{kk-lifting}
Let $A$ be a unital simple \CA\, and  let $B\subseteq A$ be a unital
separable C*-subalgebra. Suppose that $A$ contains a positive
element $b$ with $\mathrm{sp}(b)=[0, 1]$, and $A$ has Property
$\mathrm{(B1)}$ and Property $\mathrm{(B2)}$ associated with $B$ and
$\Delta_B$ and $K_0(A)$ is tracially approximately divisible.  For
any $E_0\in\mathrm{Pext(\Kone(A), \Kzero(B))}$ and
$E_1\in\mathrm{Pext(\Kzero(A), \Kone(B))}$, there exists
$\alpha\in\overline{\mathrm{Inn}}(B,A)$ such that
$\eta_0(M_\alpha)=E_0$ and $\eta_1(M_\alpha)=E_1$.
\end{thm}

\begin{proof}
Write $$\Kzero(B)_+=\{k_1, k_2,  ... , k_n, ... \}$$ and $$\Kone(B)=\{h_1, h_2,  ... , h_n,  ... \},$$ and consider the subgroups $K_n:=<k_1, ... ,k_n>$ and $H_n:=<h_1, ... ,h_n>$. Let $\{\F_i\}$ be an increasing family of finite subsets in the unit ball of $B$ with the union dense in the unit ball of $B$. Assume that for each $i$, there is a projection $p_i\in\mathrm{M}_{r_i}(\F_i)$ and unitary $z_i\in\mathrm{M}_{r_i}(\F_i)$ with $[p_i]_0=k_i$ and $[z_i]_1=h_i$.
We may assume that $r_i\le r_{i+1},$ $i\in\mathbb N.$

We assert that there are unitaries $\{u_n\}$ and diagrams

\vskip 3mm
\begin{xy}
(0,0)*{~}="0";
(20,45)*{0}="1uu";
(40,45)*{K_1(A)}="2uu";
(78,45)*{E_0}="4uu";
(110,45)*{K_0(B)}="6uu";
(130, 45)*{0}="7uu";
(20,15)*{0}="1";
(40,15)*{K_1(A)}="2";
(70,15)*{K_1(A)}="3";
(78,15)*{\oplus}="4";
(86,15)*{K_{n+1}}="5";
(110,15)*{K_{n+1}}="6";
(130, 15)*{0}="7";
(20,0)*{0}="1l";
(40,0)*{K_1(A)}="2l";
(70,0)*{K_1(A)}="3l";
(78,0)*{\oplus}="4l";
(86,0)*{K_{n}}="5l";
(110,0)*{K_{n}}="6l";
(130, 0)*{0}="7l";
(40,-15)*{\vdots}="2ll";
(70,-15)*{\vdots}="3ll";
(86,-15)*{\vdots}="5ll";
(110,-15)*{\vdots}="6ll";
(40,30)*{\vdots}="2u";
(70,30)*{\vdots}="3u";
(78,30)*{~}="4u";
(86,30)*{\vdots}="5u";
(110,30)*{\vdots}="6u";
{\ar "1";"2"};
{\ar "2";"3"};
{\ar^{\pi_n}@<.5ex> "5";"6"};
{\ar^{\theta_{n+1}^0}@<.5ex> "6";"5"};
{\ar "6";"7"};
{\ar "1uu";"2uu"};
{\ar "2uu";"4uu"};
{\ar^\pi "4uu";"6uu"};
{\ar "6uu";"7uu"};
{\ar "1l";"2l"};
{\ar "2l";"3l"};
{\ar^{\pi}@<.5ex> "5l";"6l"};
{\ar^{\theta_n^0}@<.5ex> "6l"; "5l"};
{\ar "6l";"7l"};
{\ar@{=} "2l";"2"};
{\ar@{=} "3l";"3"};
{\ar_{\iota_{n, n+1}} "5l";"5"};
{\ar_{\iota_{n, n+1}} "6l";"6"};
{\ar_{\gamma^{0}_n} "5l";"3"};
{\ar@{=} "2u";"2uu"};
{\ar "4u";"4uu"};
{\ar "6u";"6uu"};
{\ar@{=} "2ll";"2l"};
{\ar@{=} "3ll";"3l"};
{\ar "5ll";"5l"};
{\ar "6ll";"6l"};
{\ar "5ll"; "3l"};
{\ar@{=} "2";"2u"};
{\ar@{=} "3";"3u"};
{\ar "5";"5u"};
{\ar "6";"6u"};
{\ar "5"; "3u"};
\end{xy}
\vskip 3mm
\noindent and

\vskip 3mm
\begin{xy}
(0,0)*{~}="0";
(20,45)*{0}="1uu";
(40,45)*{K_0(A)}="2uu";
(78,45)*{E_1}="4uu";
(110,45)*{K_0(B)}="6uu";
(130, 45)*{0}="7uu";
(20,15)*{0}="1";
(40,15)*{K_0(A)}="2";
(70,15)*{K_0(A)}="3";
(78,15)*{\oplus}="4";
(86,15)*{H_{n+1}}="5";
(110,15)*{H_{n+1}}="6";
(130, 15)*{0}="7";
(20,0)*{0}="1l";
(40,0)*{K_0(A)}="2l";
(70,0)*{K_0(A)}="3l";
(78,0)*{\oplus}="4l";
(86,0)*{H_{n}}="5l";
(110,0)*{H_{n}}="6l";
(130, 0)*{0}="7l";
(40,-15)*{\vdots}="2ll";
(70,-15)*{\vdots}="3ll";
(86,-15)*{\vdots}="5ll";
(110,-15)*{\vdots}="6ll";
(40,30)*{\vdots}="2u";
(70,30)*{\vdots}="3u";
(78,30)*{~}="4u";
(86,30)*{\vdots}="5u";
(110,30)*{\vdots}="6u";
{\ar "1";"2"};
{\ar "2";"3"};
{\ar^{\pi_n}@<.5ex> "5";"6"};
{\ar^{\theta_{n+1}^1}@<.5ex> "6";"5"};
{\ar "6";"7"};
{\ar "1uu";"2uu"};
{\ar "2uu";"4uu"};
{\ar^\pi "4uu";"6uu"};
{\ar "6uu";"7uu"};
{\ar "1l";"2l"};
{\ar "2l";"3l"};
{\ar^\pi@<.5ex> "5l";"6l"};
{\ar^{\theta_{n}^1}@<.5ex> "6l";"5l"};
{\ar "6l";"7l"};
{\ar@{=} "2l";"2"};
{\ar@{=} "3l";"3"};
{\ar_{\iota_{n, n+1}} "5l";"5"};
{\ar_{\iota_{n, n+1}} "6l";"6"};
{\ar_{\gamma^{1}_n} "5l";"3"};
{\ar@{=} "2u";"2uu"};
{\ar "4u";"4uu"};
{\ar "6u";"6uu"};
{\ar@{=} "2ll";"2l"};
{\ar@{=} "3ll";"3l"};
{\ar "5ll";"5l"};
{\ar "6ll";"6l"};
{\ar "5ll"; "3l"};
{\ar@{=} "2";"2u"};
{\ar@{=} "3";"3u"};
{\ar "5";"5u"};
{\ar "6";"6u"};
{\ar "5"; "3u"};
\end{xy}
\vskip 3mm

\noindent such that
$$
\norm{[u_{n+1},\, a]}\leq\frac{\delta^{\mathfrak{p}}}{r_n^2\cdot 2^{n+1}}
$$
for any $a\in\mathrm{M}_{r_n}(w_n^*\F_nw_n)$, where $w_n=u_1\cdots u_n$ and $u_1=1$.
%
%
The image of each $\gamma^1_n$ lies inside $C(A)$ so that $\Lambda\circ\gamma^1_n$ is well-defined, and $$\mathrm{bott}_1(w^*_nz_iw_n, u_{n+1})=\Lambda\circ\gamma^1_n(h_i)\quad\textrm{and}\quad\mathrm{bott}_0(w^*_np_iw_n, u_{n+1})=\gamma^0_n(k_i).$$

Moreover, each partial splitting map $\theta_n^{i}$ ($i=0, 1$)  can be extended to a partial splitting map $\tilde{\theta}_n^i$ ($i=0, 1$) defined on the subgroup generated by $K_n\cup \mathcal G_0^n$ or $H_n\cup \mathcal G_1^n$, where $\mathcal G_i$ ($i=0, 1$) is the set of generators of $G_i$ ($ i=0, 1$) of Definition \ref{Pbott} with respect to ${\delta^{\mathfrak{p}}\over{r_{n+1}^2\cdot 2^{n+1}}}, {\cal F}_n, {\cal P}_0^{(n)}, {\cal P}_1^{(n)}$, and $\imath$, where ${\cal P}_0=\{[p_1], [p_2],...,[p_n]\}$ and ${\cal P}_1=\{[z_1], [z_2],...,[z_n]\}$. Denote by the subgroups generated by $K_n\cup\mathcal G_0$ and $H_n\cup\mathcal G_1$ by $\widetilde{K}_n$ and $\widetilde{H}_n$ respectively.

Assume that we have constructed the unitaries $\{u_1=1, u_2,  ... , u_{n}\}$ and the diagrams
\vskip 3mm
\begin{xy}
(0,0)*{~}="0";
(20,30)*{0}="1uu";
(40,30)*{K_1(A)}="2uu";
(78,30)*{E_0}="4uu";
(110,30)*{K_1(B)}="6uu";
(130, 30)*{0}="7uu";
(20,15)*{0}="1";
(40,15)*{K_1(A)}="2";
(70,15)*{K_1(A)}="3";
(78,15)*{\oplus}="4";
(86,15)*{K_{n}}="5";
(110,15)*{K_{n}}="6";
(130, 15)*{0}="7";
(20,0)*{0}="1l";
(40,0)*{K_1(A)}="2l";
(70,0)*{K_1(A)}="3l";
(78,0)*{\oplus}="4l";
(86,0)*{K_{n-1}}="5l";
(110,0)*{K_{n-1}}="6l";
(130, 0)*{0}="7l";
(40,-15)*{\vdots}="2ll";
(70,-15)*{\vdots}="3ll";
(86,-15)*{\vdots}="5ll";
(110,-15)*{\vdots}="6ll";
{\ar "1";"2"};
{\ar "2";"3"};
{\ar^{\pi_n}@<.5ex> "5";"6"};
{\ar^{\theta_{n}^0}@<.5ex> "6";"5"};
{\ar "6";"7"};
{\ar "1uu";"2uu"};
{\ar "2uu";"4uu"};
{\ar^\pi "4uu";"6uu"};
{\ar "6uu";"7uu"};
{\ar "1l";"2l"};
{\ar "2l";"3l"};
{\ar^\pi@<.5ex> "5l";"6l"};
{\ar^{\theta_{n-1}^0}@<.5ex> "6l";"5l"};
{\ar "6l";"7l"};
{\ar@{=} "2l";"2"};
{\ar@{=} "3l";"3"};
{\ar_{\iota_{n-1, n}} "5l";"5"};
{\ar_{\iota_{n-1, n}} "6l";"6"};
{\ar_{\gamma^{0}_{n-1}} "5l";"3"};
{\ar@{=} "2";"2uu"};
{\ar "4";"4uu"};
{\ar "6";"6uu"};
{\ar@{=} "2ll";"2l"};
{\ar@{=} "3ll";"3l"};
{\ar "5ll";"5l"};
{\ar "6ll";"6l"};
{\ar "5ll"; "3l"};
\end{xy}
\vskip 3mm

\noindent and

\vskip 3mm
\begin{xy}
(0,0)*{~}="0";
(20,30)*{0}="1uu";
(40,30)*{K_0(A)}="2uu";
(78,30)*{E_1}="4uu";
(110,30)*{K_0(B)}="6uu";
(130, 30)*{0}="7uu";
(20,15)*{0}="1";
(40,15)*{K_0(A)}="2";
(70,15)*{K_0(A)}="3";
(78,15)*{\oplus}="4";
(86,15)*{H_{n}}="5";
(110,15)*{H_{n}}="6";
(130, 15)*{0}="7";
(20,0)*{0}="1l";
(40,0)*{K_0(A)}="2l";
(70,0)*{K_0(A)}="3l";
(78,0)*{\oplus}="4l";
(86,0)*{H_{n-1}}="5l";
(110,0)*{H_{n-1}}="6l";
(130, 0)*{0}="7l";
(40,-15)*{\vdots}="2ll";
(70,-15)*{\vdots}="3ll";
(86,-15)*{\vdots}="5ll";
(110,-15)*{\vdots}="6ll";
{\ar "1";"2"};
{\ar "2";"3"};
{\ar^{\pi_{n-1}}@<.5ex> "5";"6"};
{\ar^{\theta_{n}^1}@<.5ex> "6";"5"};
{\ar "6";"7"};
{\ar "1uu";"2uu"};
{\ar "2uu";"4uu"};
{\ar^\pi "4uu";"6uu"};
{\ar "6uu";"7uu"};
{\ar "1l";"2l"};
{\ar "2l";"3l"};
{\ar^\pi@<.5ex> "5l";"6l"};
{\ar^{\theta_{n-1}^1}@<.5ex> "6l";"5l"};
{\ar "6l";"7l"};
{\ar@{=} "2l";"2"};
{\ar@{=} "3l";"3"};
{\ar_{\iota_{n-1, n}} "5l";"5"};
{\ar_{\iota_{n-1, n}} "6l";"6"};
{\ar_{\gamma^{1}_{n-1}} "5l";"3"};
{\ar@{=} "2";"2uu"};
{\ar "4";"4uu"};
{\ar "6";"6uu"};
{\ar@{=} "2ll";"2l"};
{\ar@{=} "3ll";"3l"};
{\ar "5ll";"5l"};
{\ar "6ll";"6l"};
{\ar "5ll"; "3l"};
\end{xy}
\vskip 3mm

\noindent satisfying the above assertion.

Denote by
$$
\delta_n=\Delta_B({\delta^{\mathfrak{p}}\over{r_{n+1}^2\cdot 2^{n+1}}}, {\cal F}_n, {\cal P}_0^{(n)}, {\cal P}_1^{(n)}, \textrm{ad}(w_n)\circ \imath),
$$
where ${\cal P}_0=\{[p_1], [p_2],...,[p_n]\}$ and ${\cal P}_1=\{[z_1], [z_2],...,[z_n]\}$. 
Set
$$
w_n^{(r_n)}=\text{diag}(\overbrace{w_n, w_n,...,w_n}^{r_n}),\,\,\,n=1,2, ...
$$
We note that $[p_i]=[(w_n^{(r_n)})^*p_iw_n^{(r_n)}]$ and $[z_i]=[(w_n^{(r_n)})^*z_iw_n^{(r_n)}],$ $i=1,2, ... $


Since $E_0$ and $E_1$ are pure extensions, there are partial splitting maps $\tilde{\theta}_{n+1}^0:\widetilde{K}_{n+1}\to E_0$ and $\tilde{\theta}_{n+1}^1: \widetilde{H}_{n+1}\to E_1$. Since
$\Kzero(A)$ is tracially approximately divisible, by
Lemma \ref{smalltrace}, the partial splitting map $\tilde{\theta}^1_{n+1}$
can be chosen so that for any $g\in\mathcal G_i^{n}\cup\{g_1,...,g_n\}$,
$$
\abs{\tau({\gamma}^1_{n}(g))}\leq\min\{\delta, \delta_n\},\quad \tforal \tau\in \mathrm{T}(A),
$$
where $\delta$ is the constant of Lemma \ref{commutator} (since $A$ has Property (B1)). Note that the maps $\gamma^0_{n}$ and $\gamma^1_{n}$ are defined on $\widetilde{K}_n$ and $\widetilde{H}_n$ respectively, in particular, on $G_0$ and $G_1$ respectively.

By Lemma \ref{commutator}, one has $$\gamma^1_{n}(h_i)=[U^*(t)zU(t)z^*]_1\in\Kone(\mathrm{S}(A))$$ for a unitary $z\in A$ and a path $U(t)\in\mathrm{C}([0, 1], \mathrm{U}_\infty(A))$ with $U_0=1$ and $\norm{[U(1), z]}\leq\delta^{\mathfrak{p}}$. Therefore, the map $$\Lambda\circ\gamma^1_n|_{H_n}: H_n\to\Kzero(A)$$ is well-defined, and by Lemma \ref{trace-preserve}, $$\abs{\tau(\Lambda\circ\gamma^1_n(h_i))}=\abs{\tau(\gamma^1_n(h_i))}\leq\delta_n$$ for any $\tau\in\mathrm{T}(A)$ and $1\leq i\leq n$.

Put $b_0= \gamma_n^{0}$ and $b_1=\Lambda\circ \gamma_n^1.$ By the assumption that $A$ has Property (B2) associated with $B$ and $\Delta_B$,  there is a unitary $u_{n+1}\in A$ such that
$$
\norm{[u_{n+1}, a]}\leq\frac{1}{r_n^2\cdot 2^{n+1}}
$$
for any $a\in\mathrm{M}_{r_n}(w^*_n\mathcal F_nw_n)$ and $$\mathrm{bott}_1(w_n^*z_iw_n, u_{n+1})=\Lambda\circ\gamma^1_{n}(h_i)\quad\textrm{and}\quad\mathrm{bott}_0(w_n^*p_iw_n, u_{n+1})=\gamma^0_{n}(k_i).$$ Denote by $\theta_{n+1}^i$ ($i=0, 1$) the restriction of $\tilde{\theta}_{n+1}^i$ ($i=0, 1$) to $K_{n+1}$ and $H_{n+1}$ respectively. Repeating this procedure, we get a sequence of unitaries $\{u_n\}$ and diagrams satisfying the assertion.

By Lemma \ref{decomp}, the inner automorphisms
$\{\mathrm{Ad}(u_1\cdots u_n)\}$ converge to a monomorphism $\alpha$,
and the extension $\eta_0(M_\alpha)$ and $\eta_1(M_\alpha)$ are
determined by the inductive limits of
\vskip 3mm
\xy
(0,0)*{~}="0";
(20,15)*{0}="1";
(40,15)*{K_1(A)}="2";
(70,15)*{K_1(A)}="3";
(78,15)*{\oplus}="4";
(86,15)*{K_{n+1}}="5";
(110,15)*{K_{n+1}}="6";
(130, 15)*{0}="7";
(20,0)*{0}="1l";
(40,0)*{K_1(A)}="2l";
(70,0)*{K_1(A)}="3l";
(78,0)*{\oplus}="4l";
(86,0)*{K_{n}}="5l";
(110,0)*{K_{n}}="6l";
(130, 0)*{0}="7l";
{\ar "1";"2"};
{\ar "2";"3"};
{\ar "5";"6"};
{\ar "6";"7"};
{\ar "1l";"2l"};
{\ar "2l";"3l"};
{\ar "5l";"6l"};
{\ar "6l";"7l"};
{\ar@{=} "2l";"2"};
{\ar@{=} "3l";"3"};
{\ar_{\iota_{n, n+1}} "5l";"5"};
{\ar_{\iota_{n, n+1}} "6l";"6"};
{\ar_{\tilde{\gamma}^{0}_n} "5l";"3"};
\endxy
\vskip 3mm
\noindent and
\vskip 3mm
\xy
(0,0)*{~}="0";
(20,15)*{0}="1";
(40,15)*{K_0(A)}="2";
(70,15)*{K_0(A)}="3";
(78,15)*{\oplus}="4";
(86,15)*{H_{n+1}}="5";
(110,15)*{H_{n+1}}="6";
(130, 15)*{0}="7";
(20,0)*{0}="1l";
(40,0)*{K_0(A)}="2l";
(70,0)*{K_0(A)}="3l";
(78,0)*{\oplus}="4l";
(86,0)*{H_{n}}="5l";
(110,0)*{H_{n}}="6l";
(130, 0)*{0}="7l";
{\ar "1";"2"};
{\ar "2";"3"};
{\ar "5";"6"};
{\ar "6";"7"};
{\ar "1l";"2l"};
{\ar "2l";"3l"};
{\ar "5l";"6l"};
{\ar "6l";"7l"};
{\ar@{=} "2l";"2"};
{\ar@{=} "3l";"3"};
{\ar_{\iota_{n, n+1}} "5l";"5"};
{\ar_{\iota_{n, n+1}} "6l";"6"};
{\ar_{\tilde{\gamma}^{1}_n} "5l";"3"};
\endxy
\vskip 3mm
%
\noindent respectively, where
$$
\tilde{\gamma}_n^{0}(k_i)=\tilde{\gamma}_n^{0}([p_i])=[(w^*_np_iw_n)u_{n+1}(w_n^*p_iw_n)+(1-w^*_np_iw_n)]
$$
and $$\tilde{\gamma}_n^{1}(h_i)=\tilde{\gamma}_n^{1}([z_i])=[\p(R^*(u_{n+1}, t), w_n^*z_iw_n)],$$ and therefore, $$\tilde{\gamma}_n^{0}(k_i)=\mathrm{bott}_0(w_n^*p_iw_n, u_{n+1})=\gamma_n^0(k_i),\quad 1\leq i\leq n,$$ that is, $\tilde{\gamma}_n^{0}={\gamma}_n^{0}$. For each cross map $\tilde{\gamma}_n^1$, one has $$\Lambda\circ\tilde{\gamma}_n^1(h_i)=\mathrm{bott}(w^*_nz_iw_n, R(u_{n+1}, 1))=\mathrm{bott}(w^*_nz_iw_n, u_{n+1})=\Lambda\circ\gamma_n^1(h_i),\quad 1\leq i\leq n.$$ Since $\Lambda$ is injective, we have that $\tilde{\gamma}_n^1={\gamma}_n^1$. Hence, one has that $\eta_0(M_\alpha)=E_0$ and $\eta_1(M_\alpha)=E_1$, as desired.
\end{proof}

\begin{prop}\label{Lbot}
Let $C$ be a unital AH-algebra and let $A$ be a unital separable simple \CA\, with $\mathrm{TR}(A)=0.$
Suppose that there is a unital
monomorphism $h: C\to A.$ Then $A$ has Property $\mathrm{(B2)}$ associated with $C$ for some $\Delta_C$ as described in \ref{Pbott}.
\end{prop}

\begin{proof}
Fix $\ep>0,$ finite subset ${\cal F}\subset C,$ ${\cal P}_0\subset \Kzero(C)$ and ${\cal P}_1\subset \Kone(C).$
Write $C=\varinjlim(C_n, \psi_n)$ so that $C_n$ and $\psi_n$ satisfy the conditions in 7.2 of \cite{Lin-Asy}.
Let $\Delta_B(\ep, {\cal F}, {\cal P}_0, {\cal P}_1, h)$ be $\dt$ as required by Lemma 7.5 of \cite{Lin-Asy}  for the above
$\ep,$ ${\cal F},$ ${\cal P}={\cal P}_0\cup {\cal P}_1.$
Let $n\ge 1$ be an integer in 7.5 of \cite{Lin-Asy} so we may assume that
$$
{\cal P}\subset \cup_{i=0,1}(\psi_{n, \infty})_{*i}({K}_i(C_n))
$$
and let $k(n)\ge n$ be as in 7.5 of \cite{Lin-Asy}.

Put $G_i=(\psi_{k(n), \infty})_{*i}({K}_i(C_{k(n)}),$ $i=0,1.$ In particular, $G_i$ is finitely generated and
${\cal P}_i\subset G_i,$ $i=0,1.$

Let $b_i: G_i\to {K}_{i-1}(A)$  be given,
$i=0,1.$ Write
\beq
{K}_i(C_{k(n)}\otimes \mathrm{C}(\mathbb{T}))&=&{K}_i(C_{k(n)})\oplus {\boldsymbol{\beta}}({K}_{i-1}(C_{k(n)})),\,\,\,i=0,1\andeqn\\
\underline{{K}}(C_{k(n)}\otimes \mathrm{C}(\mathbb{T}))&=&\underline{{K}}(C_{k(n)})\oplus {\boldsymbol{\beta}}(\underline{{K}}(C_{k(n)}))
\eneq
(see 2.10 of \cite{LnHomtp}).

Define $\kappa^{(0)}: \underline{{K}}(C_{k(n)})\to \underline{{K}}(A)$ by $\kappa^{(0)}=[h\circ \psi_{k(n), \infty}].$
Define $\kappa_i^{(1)}: {\boldsymbol{\beta}}({K}_i(C_{k(n)}))\to  {K}_{i-1}(A)$ by
$$
\kappa_i^{(1)}\circ {\boldsymbol{\beta}}(x)=b_{i-1}(x)
$$
for $x\in {K}_{i-1}(C_{k(n)}),$ $i=0,1.$  Since $C$ satisfies the UCT, there is $\kappa^{(1)}\in {KK}(\mathrm{S}(C_{k(n)}), A)$ such that
$\Gamma(\kappa^{(1)})=\kappa_i^{(1)}.$
Define $\kappa:\underline{{K}}(C_{k(n)}\otimes \mathrm{C}(\mathbb{T}))\to\underline{K}(A)$ by
$\kappa|_{\underline{{K}}(C_{k(n)})}=\kappa^{(0)}$ and
$\kappa|_{{\boldsymbol{\beta}}(\underline{{K}}(C_{k(n)})}=\kappa^{(1)}.$

The lemma then follows from Lemma 7.5 of \cite{Lin-Asy}.
\end{proof}

\begin{lem}\label{KL}
Let $B$ be a unital separable simple amenable C*-algebra with
$\mathrm{TR}(B)=0$ which satisfies the UCT, and let $A$ be a
unital simple \CA\, with real rank zero, stable rank one and
weakly unperforated $K_0(A).$ Suppose that $\bar{\kappa}\in
{KL}(B, A)^{++}$ with $\bar\kappa([1_B])\le [1_A]$ in
$\Kzero(A).$ Then there is a monomorphism $\alpha: B\to A$ such
that
\beq\label{KL-1}
[\alpha]=\bar{\kappa}\,\,\,\mathrm{in}\,\,\, {KL}(B,A).
\eneq
\end{lem}
\begin{proof}
It follows from the classification theorem
(\cite{Lnduke}) that $B$ is a unital simple AH-algebra with slow
dimension growth and with real rank zero. Then the lemma follows from
Theorem 4.6 of \cite{LnKT} immediately.
\end{proof}

\begin{thm}\label{Tkk}
Let $A$ and $B$ be unital simple \CA s with $\mathrm{TR}(A)=0$ and $\mathrm{TR}(B)=0$. Assume that $B$ is separable amenable and satisfies the UCT.  Then, for any $\kappa\in {KK}(B,A)^{++}$ with
$\kappa([1_B])\le [1_A]$ in $\Kzero(A)$, there is a
monomorphism $\alpha: B\to A$ such that $[\alpha]=\kappa$ in
${KK}(B,A).$
\end{thm}

\begin{proof}
Denote by $\bar{\kappa}$ the image of $\kappa$ in
${KL}(B, A)^+$. It follows from Lemma \ref{KL} that there exists
$\alpha_1\in\mathrm{Hom}(B, A)$ such that
$[\alpha_1]_{{KL}}=\bar{\kappa}$. By considering the cut-down of $A$ by $\alpha_1(1_B)$, we can regard
$B$ as a unital C*-subalgebra of $A$ with embedding $\alpha_1$. Since
$\kappa-[\alpha_1]_{{KK}}\in\mathrm{Pext}({K}_*(B),
{K}_{*+1}(A))$, by Theorem \ref{kk-lifting}, there is an
approximately inner monomorphism $\alpha_2$ of from $B$ to $A$ such that
$[\alpha_2\circ\alpha_1]_{{KK}}-[\alpha_1]_{{KK}}=
\kappa-[\alpha_1]_{{KK}}$. Then,
$\alpha:=\alpha_2\circ\alpha_1$ is the desired homomorphism.
\end{proof}

Let us recall the following theorem from \cite{Lnrange}:

\begin{thm}[Theorem 5.2 of \cite{Lnrange}]\label{LQ}
Let $C$ be a unital AH-algebra and let $A$ be a unital simple
\CA\, with $\mathrm{TR}(A)=0.$ Suppose that there is a pair $\kappa\in
{KL}_e(C,A)^{++}$ and a continuous affine map $\lambda: \mathrm{T}(A)\to
\mathrm{T}_{\mathtt{f}}(C)$ which is compatible with $\kappa.$ Then there exists a
unital monomorphism $h: C\to A$ such that
$$
[h]=\kappa\andeqn h_{\mathrm{T}}=\lambda.
$$
\end{thm}

\begin{thm}\label{TmAH}
Let  $C$ be a unital AH-algebra and let $A$ be a  unital simple \CA s with $\mathrm{TR}(A)=0.$  Suppose that
$(\kappa, \lambda)\in {KKT}(C, A)^{++}.$
Then there is a
monomorphism $\phi: C\to A$ such that $[\phi]=\kappa$ and $\phi_\mathrm{T}=\lambda.$
\end{thm}

\begin{proof}
It follows from Theorem \ref{LQ} that there exists a unital monomorphism $\psi: C\to A$ such that
$$
[\psi]_{{KL}}={\bar \kappa}\,\,\,{\rrm{in}}\,\,\,{KL}(C,A)\andeqn \phi_\mathrm{T}=\lambda.
$$
Then $\kappa-[\psi]\in \mathrm{Pext}({K}_*(C), {K}_*(A)).$ As in the proof of \ref{Tkk}, we obtain a unital monomorphism
$\alpha: \psi(C)\to A$ for which
$$
[\alpha\circ \psi]=[\kappa]\,\,\,{\rrm in}\,\,\, {KK}(C,A)
$$
and  there exists a sequence of unitaries $\{u_n\}\subset \mathrm{U}(A)$ such that
$$
\lim_{n\to\infty}{\rrm{ ad}}\, u_n\circ \phi(c)=\alpha\circ \phi(c) \tforal c\in C.
$$
Put $\phi=\alpha\circ \psi.$
Then
\beq
\lambda(\tau)(c)&=&\tau\circ \psi(c)=\lim_{n\to\infty}\tau({\rrm{ ad}}\, u_n\circ \psi(c))\\
&=&\tau(\alpha\circ \psi(c))=\tau\circ \phi(c)\tforal c\in C.
\eneq
It follows that
$$
\phi_\mathrm{T}=\lambda,
$$ as desired.
\end{proof}

\begin{rem}
It was shown in \cite{Lnrange} that there are compact metric space
$X,$ unital simple AF-algebras $A$ and ${\bar \kappa}\in
{KL}_e(\textrm{C}(X),A)^{++}$ for which there is no unital homomorphism
$h: \mathrm{C}(X)\to A$ so that $[h]={\bar \kappa}.$ Thus the information on
 $\lambda$ is essential in general. In the
case that $C$ is also simple and has real rank zero, then the map
$\lambda$ is completely determined by ${\bar \kappa}$ since
$\rho_C(\Kzero(C))$ is dense in $\mathrm{Aff}(\mathrm{T}(C))$ and
$\mathrm{T}(C)=\mathrm{T}_{\mathtt{f}}(C).$ If the C*-algebra $C$
is real rank zero and exact, without assuming the simplicity,
then $\mathrm{T}(C)=\mathrm{S}(\Kzero(C))$, and hence one can define the map $r:
\mathrm{T}(A)\to\mathrm{T}(C)$ by factoring through
$\mathrm{S}(\Kzero(A))$. It is obviously that $r$ is compatible
with $\bar\kappa$. Moreover, $r(\tau)$ is faithful on $C$ for any
$\tau\in\mathrm{T}(A)$. Indeed, if $r(\tau)(c)=0$ for some nonzero positive element
$c\in C$, then $r(\tau)(\overline{cCc})=\{0\}$. In particular,
$r(\tau)(p)=0$ for some nonzero projection $p\in\overline{cCc}$,
which contradicts the strict positivity of $\bar\kappa$.
\end{rem}

\begin{lem}\label{surj}
Let $X$ be a Banach space, and let $\{\alpha_n\}$ be a sequence of isometries. If each $\alpha_n$ is invertible, and $\displaystyle{\lim_{n\to\infty}}\alpha_n(x)$ and $\displaystyle{\lim_{n\to\infty}}\alpha_n^{-1}(x)$ exist for any $x\in X$, then $\alpha:=\displaystyle{\lim_{n\to\infty}}\alpha_n$ is invertible.
\end{lem}
\begin{proof}
Denote by $\beta=\displaystyle{\lim_{n\to\infty}\alpha^{-1}_n}$. It is clear that $\alpha$ and $\beta$ are isometries. Fix an element $x\in X$. For any $\ep>0$, there exists $N$ such that
$$\norm{\beta(x)-\alpha^{-1}_n(x)}\leq\ep\quad\textrm{and}\quad\norm{\alpha\circ\beta(x)-\alpha_n\circ\beta(x)}\leq\ep$$for any $n\geq N.$ Then
\begin{eqnarray*}
\norm{\alpha\circ\beta(x)-x}&\leq&\norm{\alpha\circ\alpha^{-1}_n(x)-x}+\norm{\alpha\circ\beta(x)-\alpha\circ\alpha^{-1}_n(x)}\\
&\leq&\norm{\alpha\circ\alpha_n^{-1}(x)-\alpha_n\circ\alpha^{-1}_n(x)}+\ep\\
&\leq&\norm{\alpha\circ\alpha_n^{-1}(x)-\alpha\circ\beta(x)}+\norm{\alpha\circ\beta(x)-\alpha_n\circ\beta(x)}+\norm{\alpha_n\circ\beta(x)-\alpha_n\circ\alpha^{-1}_n(x)}+\ep\\
&\leq&4\ep.
\end{eqnarray*}
Since $\ep$ is arbitrary, one has that $\alpha\circ\beta(x)=x$, and hence $\alpha\circ\beta=\textrm{id}$. The argument same as above shows that $\beta\circ\alpha=\textrm{id}$. Therefore, $\alpha$ is invertible, as desired.
\end{proof}

\begin{defn}\label{invKL}
Let $A$ be a unital C*-algebra. Denote by $KK^{-1}_e(A, A)^{++}$ the
set of those elements $\kappa\in {KK}_e(A, A)^{++}$ such that
$\kappa$ induces an ordered isomorphism between $\Kzero(A)$ and
isomorphisms between $\underline{K}(A)$.
\end{defn}

\begin{cor}\label{Caut0}
Let $A$ be a unital separable amenable simple \CA\, with $\mathrm{TR}(A)=0$ and satisfies the UCT. Then, for any $\kappa\in {KK}_e^{-1}(A, A)^{++}$,
 there exists an automorphism $\alpha\in \mathrm{Aut}(A)$ such that
$$
[\alpha]=\kappa.
$$
\end{cor}
\begin{proof}
By Theorem \ref{Tkk} and its remark , there is a monomorphism
$\alpha=\alpha_2\circ\alpha_1$ such that $[\alpha]=\kappa$ in $KK(A,
A)$. Let us show that $\alpha$ can be chosen to be an automorphism.
Let us first show that $\alpha_1$ can be chosen to be an
automorphism. We first choose a unital monomorphism $\alpha_1':A\to
A$ so that $[\alpha']=\kappa$ by \ref{Tkk} and its remark.

 By the UCT, there is $\kappa_1\in KK_e^{-1}(A,A)^{++}$ such that
$\kappa_1\times \kappa=\kappa\times \kappa_1=[{\rrm{ id}}_A].$ So
there is a unital monomorphism $\beta: A\to A$ such that
$[\beta]=\kappa_1.$ Then $[\beta\circ \alpha']=[\alpha'\circ
\beta]=[{\rrm{id}}_A].$ By the uniqueness theorem (2.3 of
\cite{LinTAF2}), both $\beta\circ \alpha'$ and $\alpha'\circ \beta$
are approximately unitarily equivalent to the identity.  Then, by a
standard intertwining argument of Elliott, one obtains an
isomorphism $\alpha_1: A\to A$ which is approximately unitarily
equivalent to $\alpha'$ (see for example Theorem 3.6 of
\cite{LinTAF2}).  So in particular,
$\overline{[\alpha_1]}=\overline{\kappa}$ in $KL(A,A).$ (This also
follows from the proof of Theorem 3.7 of \cite{LinTAF2} that there
is an isomorphism $\alpha_1: A\to A$ such that
$[\alpha_1]=\bar{\kappa}$ in $KL(A, A),$ using that fact that $A$ is
pre-classifiable in the sense of \cite{LinTAF2}, see Theorem 4.2 of
\cite{Lnduke}). .

Consider the map $\alpha_2$.
Note that if $A=B$, then, in the proof of Theorem \ref{kk-lifting},
the union of finite subsets $\mathcal{F}_n$ is dense in $A$, and the inner
automorphisms $\{\textrm{Ad}(u_1\cdots u_n)\}$ satisfy Lemma \ref{surj}.
Therefore, by Lemma \ref{surj}, the monomorphism
$\alpha_2=\displaystyle{\lim_{n\to\infty}}\textrm{Ad}(u_1\cdots u_n)$ is an automorphism of $A$.

Therefore, the map $\alpha=\alpha_2\circ\alpha_1$ is an automorphism of $A$.
\end{proof}

\section{Rotation maps}\label{rotation}

\begin{lem}\label{approx}
Let $H$ be a finitely generated abelian group, and let $A$ be a
C*-algebra with $\rho_A(\Kzero(A))$ dense in
$\mathrm{Aff}(\mathrm{T}(A))$. Let $\psi\in\mathrm{Hom}(H,
\mathrm{Aff}(\mathrm{T}(A)))$. Fix $\{g_1,  ... , g_n\}\subseteq
H$. Then, for any $\ep>0$, there is a homomorphism
$h:H\to\Kzero(A)$ such that
$$\abs{\psi(g_i)-\rho_A(h(g_i))}<\ep$$ for any $1\leq i\leq n$.
\end{lem}
\begin{proof}
Since the map $\psi$ factors through $H/H_{\mathrm{Tor}}$, we may
assume that $H=\bigoplus^k\Int$ for some $k$ and prove the lemma
for an $\{g_i\}$ replaced by the standard base $\{e_i;\
i=1, ... ,k\}$. Since $\rho_A(\Kzero(A))$ is dense in
$\mathrm{Aff}(\mathrm{T}(A))$, there exist $a_1,  ... ,
a_n\in\Kzero(A)$ such that
$$\abs{\psi(e_i)-\rho_A(a_i)}<\ep,\quad i=1, ... ,k.$$ The the
map $h: H\to\Kzero(A)$ defined by
$$e_i\mapsto a_i$$ satisfies the lemma.
\end{proof}

\begin{thm}\label{rotation-maps}
Let $A$ be a unital simple \CA\ with $\rho_A(\Kzero(A))$ dense in $\mathrm{Aff}(\mathrm{T}(A))$. Assume that there is a positive element $b\in A$ with $\mathrm{sp}(b)=[0, 1]$. Let $B\subseteq A$ be a unital  C*-subalgebra and denote by $\iota$ the inclusion map. Suppose that $A$ has Property $\mathrm{(B1)}$ and Property $\mathrm{(B2)}$ associated with $B$ and $\Delta_B$.
For any $\psi\in\mathrm{Hom}(\Kone(B), \mathrm{Aff}(\mathrm{T}(A)))$, there exists $\alpha\in\overline{\mathrm{Inn}}(B, A)$ such that there are maps $\theta_i:{K}_i(B)\to{K}_i(M_\alpha)$ with $\pi_0\circ\theta_i=\mathrm{id}_{{K}_i(B)}$, $i=0, 1$, and  the rotation map $R_{\iota, \alpha}: \Kone(M_\alpha)\to\mathrm{Aff}(\mathrm{T}(A))$ is given by
$$R_{\iota, \alpha}(c)=\rho_A(c-\theta_1([\pi_0]_1(c)))+\psi([\pi_0]_1(c)),\quad \forall c\in\Kone(M_\alpha(A)).$$ In other words, $$[\alpha]=[\iota]$$ in ${KK}(B, A)$, and the rotation map $R_{\iota, \alpha}: \Kone(M_\alpha)\to\mathrm{Aff}(\mathrm{T}(A))$ is given by
$$R_{\iota, \alpha}(a, b)=\rho_A(a)+\psi(b)$$ for some identification of $\Kone(M_\alpha)$ with $\Kzero(A)\oplus\Kone(B)$.
\end{thm}

\begin{proof}
Write $$\Kone(B)=\{g_1, g_2,  ... , g_n,  ... \},$$
$$\Kzero(B)_+=\{k_1, k_2,  ... , k_n,  ... \},$$ and denote by
$H_n=<g_1,  ... , g_n>,$ { the subgroup generated by
$\{g_1, ... ,g_n\}$}, and $K_n=<k_1, ... , k_n>,$ {
the subgroup generated by $\{k_1, ... ,k_n\}$}. Let $\{\F_i\}$
be an increasing family of finite subsets with the union dense in
$B$. Assume that for each $i$, there is a unitary $z_i$ and a
projection $p_i$ in $\mathrm{M}_{r_i}(\F_i)$ for some
natural number $r_i$ such that $[z_i]_1=g_i$ and $[p_i]=k_i$.
Without loss of generality, we may assume that $r_i\le r_{i+1},$
$i\in\mathbb N.$  In what follows, if $v\in A,$ by $v^{(m)},$ we mean
$v^{(m)}={\rrm{diag}}(\overbrace{v,v, ... ,v}^m).$

We assert that there are maps $\{h_i: H_i\to\Kzero(A);\ i\in\mathbb N\}$ and unitaries $\{u_i;\ i\in\mathbb N\}$ such that for any $n$, if denoted by $w_n=u_1\cdots u_{n-1}$ (assume $u_{-1}=u_0=1$), one has
\begin{enumerate}
\item For any $x\in\{g_1,  ... , g_i\}\cup\mathcal G^n_1$,
$$\abs{\rho_A\circ h_n(x)-\psi(x)}<\frac{\delta_n\delta}{2^{n}}$$ where $$\delta_n=\Delta_B(\frac{\delta^{\mathfrak{p}}}{r_n^2\cdot 2^{n+1}}, {\mathcal F_n}, {\mathcal P_n^{(0)}}, {\mathcal P_n^{(1)}}, \mathrm{Ad}(w_{n-1})\circ\iota)$$ for $\mathcal P_n^{(0)}=\{k_1, k_2, ... , k_n\}$ and $\mathcal P_n^{(1)}=\{g_1, g_2, ... , g_n\}$, $\delta$ the constant of Lemma \ref{commutator} (since $A$ has Property (B1)), and $\mathcal G^n_1$ is the set of generators of $G_1$ of Definition \ref{Pbott} with respect to $\frac{\delta^{\mathfrak{p}}}{r_n^2\cdot 2^{n+1}}, {\mathcal F_n}, {\mathcal P_n^{(0)}}, {\mathcal P_n^{(1)}},$ and  $\iota)$.

\item For any $a\in w_{n-1}^*{\mathcal F_n}w_{n-1}$,
$$\norm{[u_{n}, a]}\leq\frac{\delta^{\mathfrak{p}}}{r_n^2\cdot2^{n+1}}$$ where $w_n=u_1 ...  u_n$, and moreover, the image of $\phi_n:=h_{n+1}|_{H_n}-h_n$ lies inside $C(A)$ so that $\Lambda\circ\phi_n$ is well-defined, and $$\mathrm{bott}_1((w^{(r_{n})}_{n-1})^*z_iw^{(r_{n})}_{n-1}, u_{n})=\Lambda\circ\phi_{n}(g_i)\quad{\rrm{and}}\quad \mathrm{bott}_0((w^{(r_{n})}_{n-1})^*p_iw^{(r_{n})}_{n-1}, u_{n})=0$$ for any $i=1,  ... , n$.
\end{enumerate}

If $n=1$, since $\rho_A(\Kzero(A))$  is dense in $\mathrm{Aff}(\mathrm{T}(A))$, by Lemma \ref{approx}, there is a map $h_1: <\{g_1\}\cup\mathcal G_1>\to\Kzero(A)$ such that for any $x\in \{g_1\}\cup\mathcal G_1$
$$
\abs{\psi(x)-\rho_A(h_1(x))}<\frac{\delta_1\delta}{2},
$$
and a map $h_2: <\{g_1, g_2\}\cup\mathcal G_2>\to\Kzero(A)$ such that
$$\abs{\psi(x)-\rho_A(h_2(x))}\leq\frac{\delta_1\delta}{2^2}$$ for any
$x\in \{g_1, g_2\}\cup\mathcal G_2$. For the function $$\phi_1=h_2|_{H_1}-h_1,$$ we have that $\abs{\tau(\phi_1(g_1))}<\min\{\delta, \delta_1\}$ for any $\tau\in\mathrm{T}(A)$ and hence
$\phi_1(g_1)\in C(A)$ by Lemma \ref{commutator}. By Lemma
\ref{trace-preserve},
$$\abs{\tau(\Lambda\circ\phi_1(x))}=\abs{\tau(\phi_1(x))}<\delta_1$$
for any $\tau\in\mathrm{T}(A)$ and any $x\in\mathcal G_1$. 
Since $A$ has Property (B2) associated with $B$
and $\Delta_B$, there is a unitary $u_1\in \mathrm{U}(A)$ such
that
$$\norm{[u_1, a]}\leq\frac{\delta^{\mathfrak{p}}}{r_1^2\cdot
2^2},\quad\forall a\in\F_1,$$ and
$$
\mathrm{bott}_1(z_1, u_1)={{\Lambda}}\circ\phi_1(g_1)\quad\textrm{and}\quad \mathrm{bott}_0(p_1, u_1)=0.
$$

Assume that we have constructed the maps $\{h_i: H_i\to\Kzero(A);\ i=1, ... , n\}$ and unitaries $\{u_i;\ i=1,  ... , n-1\}$ satisfying the assertion above. By by Lemma \ref{approx},  there is a function $h_{n+1}:<\{g_1,...,g_{n+1}\}\cup\mathcal G_1^n\cup\mathcal G_1^{n+1}>\to\Kzero(A)$ such that for any $x\in\{g_1,  ... , g_{n+1}\}\cup\mathcal G_1^n\cup\mathcal G_1^{n+1}$, $$\abs{\rho_A\circ h_{n+1}(x)-\psi(x)}<\frac{\delta_{n+1}\delta}{2^{n+1}}$$ where $$\delta_{n+1}=\Delta_B(\frac{{{\delta^{\mathfrak{p}}}}}{r_n^2\cdot 2^{n+2}}, {\mathcal F_{n+1}}, {\mathcal P_{n+1}^{(0)}}, {\mathcal P_{n+1}^{(1)}}, \mathrm{Ad}(w_{n})\circ\iota)$$ for $\mathcal P_{n+1}^{(0)}=\{k_1, k_2, ... , k_{n+1}\}$ and $\mathcal P_{n+1}^{(1)}=\{g_1, g_2, ... , g_{n+1}\}$. 

Recall that $\phi_n=h_{n+1}|_{H_n}-h_n$. Then, a direct calculation shows that for any $\tau\in\mathrm{T}(A)$, $$\abs{\tau(\Lambda\circ\phi_{n})(x)}=\abs{\tau(\phi_n(x))}<\min\{\delta, \delta_{n+1}\}$$ for any $x\in\mathcal G_1^n$. Since $A$ Property (B2) associated with $B$ and $\Delta_B$,  there is a unitary $u_{n}\in \mathrm{U}(A)$ such that $$\norm{[u_{n}, w_{n-1}^*aw_{n-1}]}\leq\frac{\delta^{\mathfrak{p}}}{r_n^2\cdot 2^{n+1}},\quad\forall a\in\F_n,$$ and
$$
\mathrm{bott}_1(\mathrm{Ad}(w_{n-1})\circ\iota, u_{n})|_{{\cal P}_{n}^{(1)}}=
{{\Lambda}}\circ\phi_n\quad\textrm{and}\quad \mathrm{bott}_0(\mathrm{Ad}(w_{n-1})\circ\iota, u_{n})|_{{\cal P}_{n}^{(0)}}=0.
$$ This proves the assertion.

By Lemma \ref{decomp}, $\textrm{Ad}(w_n)$ converge to a monomorphism $\alpha: B\to A$. Moreover, the extension $\eta_0(M_{\alpha})$ is trivial, and $\eta_1(M_{\alpha})$ is determined by the inductive limit of
\vskip 3mm
\xy
(0,0)*{~}="0";
(20,15)*{0}="1";
(40,15)*{K_0(A)}="2";
(70,15)*{K_0(A)}="3";
(78,15)*{\oplus}="4";
(86,15)*{H_{n+1}}="5";
(110,15)*{H_{n+1}}="6";
(130, 15)*{0}="7";
(20,0)*{0}="1l";
(40,0)*{K_0(A)}="2l";
(70,0)*{K_0(A)}="3l";
(78,0)*{\oplus}="4l";
(86,0)*{H_{n}}="5l";
(110,0)*{H_{n}}="6l";
(130, 0)*{0}="7l";
{\ar "1";"2"};
{\ar "2";"3"};
{\ar "5";"6"};
{\ar "6";"7"};
{\ar "1l";"2l"};
{\ar "2l";"3l"};
{\ar "5l";"6l"};
{\ar "6l";"7l"};
{\ar@{=} "2l";"2"};
{\ar@{=} "3l";"3"};
{\ar_{\iota_{n, n+1}} "5l";"5"};
{\ar_{\iota_{n, n+1}} "6l";"6"};
{\ar_{{\gamma}_n} "5l";"3"};
\endxy
\vskip 3mm
%
\noindent where $\gamma_n(g_i)=[\p(R^*(u_{n+1}, t), w_n^*z_iw_n)]$. However, since $$\Lambda\circ\gamma_n(g_i)=\mathrm{bott}_1(\textrm{Ad}(w_{n-1})\circ\iota, u_{n})(g_i)=\Lambda\circ\phi_n(g_i),$$ by the injectivity of $\Gamma$, we have that $\gamma_n=\phi_n$.

We assert that $\eta_1(M_\alpha)$ is also trivial. For any $n$, define a map $\theta'_n:H_{n}\to\Kone(M_{\alpha})$ by $$\theta'_n(g)=(h_{n}(g), g)).$$ We then have
\begin{eqnarray*}
&&\theta'_{n+1}\circ\iota_{n, n+1}(g)-\theta'_n(g)\\
&=&(h_{n+1}\circ\iota_{n, n+1}(g), \iota_{n, n+1}(g))-(h_n(g)+\phi_n(g), \iota_{n, n+1}(g))\\
&=&(h_{n+1}\circ\iota_{n, n+1}(g), \iota_{n, n+1}(g))-(h_n(g)+(h_{n+1}(\iota_{n, n+1}(g))-h_n(g)), \iota_{n, n+1}(g))\\
&=&0,
\end{eqnarray*}
and hence $(\theta'_n)$ define a homomorphism $\theta_1:\Kone(B) \to K_1(M_\alpha)$. Moreover, since $\pi\circ\theta_1=\mathrm{id}_{K_1(B)}$, the extension $\eta_1(M_\alpha)$ splits. Therefore, $[\alpha]=[\iota]$ in ${KK}(B, A)$.

Let us calculate the corresponding rotation map.  Pick $[z_i]=g_i$ for some $g_i\in H_{m}$.
To simplify notation, without loss of generality, we will use $\alpha$ for $\alpha\otimes {\rrm{id}}_{\mathrm{M}_{r}(\Comp)}$
on $B\otimes \mathrm{M}_r(\Comp)$ for any integer $r\ge 1.$
Take a path $v(t)$ in $\mathrm{M}_{2r_i}(A)$ from $z_i$ to $\alpha(z_i)$ as follows
\begin{displaymath}
v(t)=R^*(2t)
\left(
\begin{array}{cc}
1& 0\\
0 & (w_m^{(r(i))})^*
\end{array}
\right)
R(2t)
\left(
\begin{array}{cc}
z_i& 0\\
0 & 1
\end{array}
\right)
R^*(2t)
\left(
\begin{array}{cc}
1& 0\\
0 & w_m^{(r(i))}
\end{array}
\right)
R(2t)
\end{displaymath}
for $t\in [0,1/2]$ and
$$
v(t) =(w_m^{(r(i))})^*z_iw_m^{(r(i))}\exp((2t-1)c_m)
$$
for $t\in [1/2,1],$ where
\begin{displaymath}
R(t)=\left(
\begin{array}{cc}
\cos(\frac{\pi t}{2})& \sin(\frac{\pi t}{2} )\\
-\sin(\frac{\pi t}{2}) & \cos(\frac{\pi t}{2})
\end{array}
\right)
\end{displaymath}
and $c_m=\log((w_m^{(r(i))})^*z^*_iw_m^{(r(i))}{\alpha}(z_i))$.

Then, for any $\tau\in\mathrm{T}(A)$, we have
\begin{eqnarray*}
\frac{1}{2\pi i}\int_0^1\tau(\dot{v}(t)v^*(t))\mathrm{d}t&=&\tau(\log((w_m^{(r(i))})^*z^*_iw_m^{(r(i))}{{\alpha}}(z_i)))\\
&=&\lim_{n\to\infty} \tau(\log((w_m^{(r(i))})^*z^*w_m^{(r(i))}{\rrm{Ad}}w_n^{(r(i))}(z_i)))\\
&=& \lim_{n\to\infty} \sum_{k=0}^n\tau(\log(((w_{m+k-1}^{(r(i))})^*z^*_iw_{m+k-1}^{(r(i))}) ((w_{m+k}^{(r(i))})^*z_iw_{m+k}^{(r(i))})))\\
&=&\sum_{k=0}^{\infty}\tau(\log(((w_{m+k-1}^{(r(i))})^*z^*_iw_{m+k-1}^{(r(i))}) ((w_{m+k}^{(r(i))})^*z_iw_{m+k}^{(r(i))})))\\
&=&\sum_{k=0}^{\infty}\tau(\log(((w_{m+k-1}^{(r(i))})^*z^*_iw_{m+k-1}^{(r(i))})(u_{m+k}^{(r(i))} ((w_{m+k-1}^{(r(i))})^*z_iw_{m+k-1}^{(r(i))})u_{m+k}^{(r(i))}))\\
&=&\sum_{k=0}^{\infty}\tau(\mathrm{bott}_1((w_{m+k-1}^{(r(i))})^*z^*_iw_{m+k-1}^{(r(i))}, u_{m+k})).
\end{eqnarray*}
Then, by the choice of $\{u_n\}$, we have
\begin{eqnarray*}
\frac{1}{2\pi i}\int_0^1\tau(\dot{v}(t)v^*(t))\mathrm{d}t&=&\sum_{k=0}^{\infty}\tau(\Lambda\circ\phi_{m+k}(g_i))\\
&=&\sum_{k=0}^{\infty}\tau(\phi_{m+k}(g_i))\\
&=&\sum_{k=0}^{\infty}\tau(h_{m+k+1}(g_i)-h_{m+k}(g_i))\\
&=&\psi(g_i)-\tau(h_{m}(g_i)).
\end{eqnarray*}
Thus,
\begin{align*}
R_{\alpha_1}([v])&=\psi(g_i)-\rho_A\circ h_m(g_i)\\
&=\psi([\pi_0(v)])+\rho_A(-h_m(g_i))\\
&=\psi([\pi_0(v)])+\rho_A((0, g_i)-\theta'_m(g_i))\\
&=\psi([\pi_0(v)])+\rho_A([v]-\theta_1([\pi_0(v)])),
\end{align*}
as desired. 
\end{proof}

\begin{cor}
Let $A$ be a unital simple \CA\ with $\mathrm{TR}(A)=0$ and let $B\subseteq A$ be an AH-algebra and denote by $\iota$ the inclusion map. For any $\psi\in\mathrm{Hom}(\Kone(B), \mathrm{Aff}(\mathrm{T}(A)))$, there exists $\alpha\in\overline{\mathrm{Inn}}(B, A)$ such that there are maps $\theta_i:{K}_i(B)\to{K}_i(M_\alpha)$ with $\pi_0\circ\theta_i=\mathrm{id}_{{K}_i(B)}$, $i=0, 1$, and  the rotation map $R_{\iota, \alpha}: \Kone(M_\alpha)\to\mathrm{Aff}(\mathrm{T}(A))$ is given by
$$R_{\iota, \alpha}(c)=\rho_A(c-\theta_1([\pi_0]_1(c)))+\psi([\pi_0]_1(c)),\quad \forall c\in\Kone(M_\alpha(A)).$$
In other words, $$[\alpha]=[\iota]$$ in ${{KK}}(B, A)$ and the rotation map $R_{\iota,\alpha}: \Kone(M_\alpha)\to\mathrm{Aff}(\mathrm{T}(A))$ is given by
$$R_{\iota, \alpha}(a, b)=\rho_A(a)+\psi(b)$$ for some identification of $\Kone(M_\alpha)$ with $\Kzero(A)\oplus\Kone(B)$.
\end{cor}
\begin{proof}
It follows directly from Lemma 5.2 of \cite{Lin-Asy}, Theorem \ref{rotation-maps}, and Proposition \ref{Lbot}.
\end{proof}

\begin{defn}\label{DR0}
Fix two unital \CA s $A$ and $B$ with
$\mathrm{T}(A)\not=\varnothing.$ Let ${\cal R}_0$ be the subset of
$\textrm{Hom}(K_1(B), \text{Aff}(\mathrm{T}(A)))$ consisting  of
those homomorphisms $h\in \textrm{Hom}(K_1(B),
{\text{Aff}}(\mathrm{T}(A)))$ such that there exists a homomorphism
$d: K_1(B)\to K_0(A)$ such that
$$
h=\rho_A\circ d.
$$
Then ${\cal R}_0$ is clearly a subgroup of $\textrm{Hom}(K_1(B),
\text{Aff}(\mathrm{T}(A))).$

\end{defn}

The following is a variation of Theorem 10.7 of \cite{Lin-Asy}.

\begin{thm}\label{Lr0}
Let $C$ be a unital AH-algebra and let $A$ be a unital separable simple \CA\,
with $\mathrm{TR}(A)=0.$
Suppose that $\phi,\, \psi:C \to A$ are two unital monomorphisms such that
$$[\phi]=[\psi]$$
in ${KK}(C,A)$ {and,} for all $\tau\in
\mathrm{T}(A),$
$$\tau\circ \phi=\tau\circ \psi.$$
Suppose also that there exists a homomorphism
$\theta\in\mathrm{Hom}(\Kone(C), \Kone(M_{\phi,
\psi}))$ with $(\pi_0)_{*1}\circ
\theta=\mathrm{id}_{\Kone(C)}$ such that
\beq\label{Lr0-1}
R_{\phi, \psi}\circ \theta\in {\cal R}_0.
\eneq
Then $\phi$ and $\psi$ are asymptotically unitarily equivalent.
\end{thm}

\begin{proof}
We may write
$$
\Kone(M_{\phi, \psi})=\Kzero(A)\oplus \theta(\Kone(C)).
$$
Let $h=R_{\phi, \psi}\circ \theta.$ If there is a {homomorphism}
$d:K_1(C)\to K_0(A)$ such that $h=\rho_A\circ d,$ define
$$
\theta'=\theta-d.
$$
Note that $\theta$ is a homomorphism, and $(\pi_0)_{*1}\circ \theta={\mathrm
{id}}_{K_1(C)}.$  Then
$$
R_{\phi, \psi}\circ \theta'=R_{\phi, \psi}\circ \theta-\rho_A\circ
d=0.$$



Since $[\phi]=[\psi]$ in ${KK}(C,A),$ there exists an element $\theta_0\in \mathrm{Hom}_{\Lambda}(\underline{{K}}(C), \underline{{K}}(M_{\phi, \psi}))$ such that
$[\pi_0]\circ [\theta_0]=[\rrm{id}_{\underline{{K}}(C)}].$
Define $\theta_0': K_i(C)\to {K}_{i}(M_{\phi, \psi})$ by $\theta_0'|_{{K}_1(C)}=\theta'$ and
$\theta_0'|_{{K}_0(C)}=\theta_0|_{{K}_0(C)}.$ By the UCT, there exists $\theta_0''\in {KL}(C,M_{\phi, \psi})$
such that $\Gamma(\theta_0'')=\theta_0',$  where $\Gamma$ is the map from ${KL}(C, M_{\phi, \psi})$ onto
$\mathrm{Hom}({K}_*(C), {K}_*(M_{\phi, \psi})).$
We will use the identification
$$
{KL}(\mathrm{S}C, M_{\phi, \psi})=\mathrm{Hom}_{\Lambda}(\underline{{K}}(\mathrm{S}C), \underline{{K}}(M_{\phi,\psi})).
$$ (see \cite{DL}).
Put
$$
x_0=[\pi_0]\circ \theta_0''-[\rrm{id}_{\underline{{K}}(C)}].
$$
Then $\Gamma(x_0)=0.$
Define $\theta_1=\theta_0''-\theta_0\circ x_0\in \mathrm{Hom}_{\Lambda}(\underline{{K}}(C), \underline{{K}}(M_{\phi,\psi})).$ Then one computes that
\beq\label{Lr0-3}
[\pi_0]\circ \theta_1 &=& [\pi_0]\circ \theta_0''-[\pi_0]\circ \theta_0\circ x_0\\
&=&([\rrm{id}_{\underline{{K}}(C)}]+x_0)-[\rrm{id}_{\underline{{K}}(C)}]\circ x_0\\
&=&[\rrm{id}_{\underline{{K}}(C)}]+x_0-x_0=[\rrm{id}_{\underline{{K}}(C)}].
\eneq
Moreover,
$$
\Gamma(\theta_1)|_{{K}_1(C)}=\theta'.
$$
In particular,
$$
R_{\phi, \psi}\circ (\theta_1)_{{K}_1(C)}=0.
$$
It follows from Theorem 10.7 of \cite{Lin-Asy} that $\phi$ and $\psi$ are asymptotically unitarily equivalent.
\end{proof}

\begin{rem}\label{R2}
In \ref{Lr0}, the condition that such $\theta$ exists is also necessary.
This follows from Theorem 9.1 of \cite{Lin-Asy}.
\end{rem}

\begin{defn}\label{KKM}
Let $A$ be a unital C*-algebra, and let $C$ be a unital separable \CA.
Denote by $\mathrm{Mon}_{asu}^e(C,A)$ the set of asymptotically
unitary equivalence classes of unital monomorphisms.

Denote by ${\boldsymbol{K}}$ the map from
$\textrm{Mon}_{asu}^e(C, A)$ into ${KK}_e(C,A)^{++}$ defined by
$$
\phi\mapsto [\phi]\tforal \phi\in \mathrm{Mon}_{asu}^e(C,A).
$$
Let $\kappa\in {KK}_e(C,A)^{++}.$ Denote by $\langle \kappa \rangle$
the classes of $\phi\in \mathrm{Mon}_{asu}^e(C,A)$ such that
${\boldsymbol{ K}}(\phi)=\kappa.$

Denote by ${\widetilde{\boldsymbol{ K}}}$ the map from
$\textrm{Mon}_{asu}^e(C, A)$ into ${KKT}(C,A)^{++}$ defined by
$$
\phi\mapsto ([\phi], \phi_T)\tforal \phi\in
\mathrm{Mon}_{asu}^e(C,A).
$$
Denote by $\langle \kappa, \lambda \rangle $ the classes of
$\phi\in \mathrm{Mon}_{asu}^e(C,A)$ such that
${\widetilde{\boldsymbol{ K}}}(\phi)=(\kappa,\,\lambda).$

\end{defn}

\begin{thm}\label{MT2}
Let $C$ be a unital AH-algebra and let $A$ be a unital separable
simple \CA\, with $\mathrm{TR}(A)=0.$ Then the map
${\widetilde{\boldsymbol{K}}}: \mathrm{Mon}_{asu}^e(C,A)\to
{KKT}(C,A)^{++}$ is surjective. Moreover, for each $(\kappa,
\lambda)\in {KKT}(C,A)^{++}$, there exists a bijection
$$
\eta: \langle \kappa,\lambda \rangle \to \mathrm{Hom}({K}_1(C),
\aff(\mathrm{T}(A)))/{\cal R}_0.
$$
\end{thm}

\begin{proof}
It follows from \ref{TmAH} that ${\widetilde{\boldsymbol{K}}}$ is surjective.

Fix a pair $(\kappa, \lambda)\in {KKT}(C,A)^{++}$ and fix
a unital monomorphism $\phi: C\to A$ such that $[\phi]=\kappa$ and $\phi_\mathrm{T}=\lambda.$
For any homomorphism $\gamma\in \mathrm{Hom}({K}_1(C), \aff(\mathrm{T}(A)))$, it follows from \ref{rotation-maps} that there exists a unital monomorphism $\alpha\in {\overline{{\rrm{Inn}}}}(\phi(C), A)$ with $[\alpha\circ \phi]=[\phi]$ in ${KK}(C,A)$ such that there exists a homomorphism $\theta: {K}_1(C)\to M_{\phi, \alpha\circ \phi}$
with $(\pi_0)_{*1}\circ \theta={\rrm{id}}_{{K}_1(C)}$ such that $R_{\phi, \alpha\circ \phi}\circ \theta=\gamma.$
Let $\psi=\alpha\circ \phi.$  Then $R_{\phi, \psi}\circ \theta=\gamma.$ Note also since $\alpha\in {\overline{\rrm{Inn}}}(\phi(C), A),$
$\psi_\mathrm{T}=\phi_\mathrm{T}.$ In particular, ${\widetilde{\boldsymbol{K}}}(\psi)={\widetilde{\boldsymbol{K}}}(\phi).$

Suppose that $\theta': {K}_1(A)\to {K}_1(M_{\phi, \psi})$ such that $(\pi_0)_{*1}\circ \theta'=\rrm{id}_{{K}_1(C)}.$
Then

$$
(\theta'-\theta)(K_1(C))\subset  K_0(A).
$$
It follows that $R_{\phi, \psi}\circ \theta'-R_{\phi,
\psi}\circ \theta\in {\cal R}_0.$

Thus we obtain a well-defined map ${\eta: \langle [\phi],
 \phi_T\rangle \to \mathrm{Hom}({K}_1(A),
\text{Aff}(\text{T}(A)))/{\cal R}_0.}$ From what we have proved, the map
$\eta$ is surjective.
%
%
%

Suppose that $\phi_1, \phi_2: C\to A$ are two unital monomorphisms
such that $\phi_1, \phi_2\in \langle [\phi], \phi_T\rangle$ and
$$
R_{\phi, \phi_1}\circ \theta_1-R_{\phi,\phi_2}\circ \theta_2\in
{\cal R}_0,
$$
where $\theta_1: K_1(C)\to K_1(M_{\phi, \phi_1})$ and $\theta_2:
K_1(C)\to K_1(M_{\phi, \phi_2})$ are homomorphisms such that
$(\pi_0)_{*1}\circ \theta_1=(\pi_0)_{*1}\circ \theta_2={\mathrm
{id}}_{K_1(C)},$ respectively.
Thus there is $\dt: K_1(C)\to K_0(A)$ such that
$$
R_{\phi, \phi_2}\circ \theta_2-R_{\phi,\phi_1}\circ
\theta_1=\rho_A\circ \dt.
$$

Fix $z\in K_1(C)$ and let $u\in M_k(C)$ such that $[u]=z$ in
$K_1(C).$ Let $U(\theta_1,z)(t)\in \mathrm{U}(\mathrm{M}_N(M_{\phi,
\phi_1}))$ be a unitary represented by $\theta_1(z).$ We may assume
that $N\ge k.$ It is easy to see that $U(\theta_1,z)(t)$ is
homotopic to a unitary $V(\theta_1,z)(t)\in
\mathrm{U}(\mathrm{M}_N(M_{\phi, \phi_1}))$ with
$V(\theta_1,z)(0)={\rrm{diag}}(\phi(u), 1_{N-k})$ and
$V(\theta_2,z)(1)={\rrm{diag}}(\phi_1(u), 1_{N-k})$ (by enlarging
the size of matrices if necessary). In particular,
$[V(\theta_1,z)]=\theta_1(z)$ in $K_1(M_{\phi, \phi_1}).$ Similarly,
we may assume that there is a unitary $V(\theta_2,z)\in
\mathrm{U}(\mathrm{M}_N(M_{\phi, \phi_2}))$ with
$V(\theta_2,z)(0)={\rrm {diag}}(\phi(u), 1_{N-k})$ and $V(\theta_2,
z)(1)={\rrm {diag}}(\phi_2(u),1_{N-k})$ which is represented by
$\theta_2(z)$ in $K_1(M_{\phi, \phi_2}).$

Now define
\beq\label{New}
V(\theta_3,z)(t)=\begin{cases}
V(\theta_1,z)(1-2t)\,\,\,\text{for}\,\,\,
t\in [0,1/2]\\
   V(\theta_2,z)(2t-1) \,\,\,\text{for}\,\,\, t\in (1/2,1],
\end{cases}
\eneq
Note that $V(\theta_3,z)\in \mathrm{U}(\mathrm{M}_N(M_{\phi_1,\phi_2})).$


Define the homomorphism $\theta_3: K_1(C)\to K_1(M_{\phi_1, \phi_2})$ by
$$
\theta_3(z)=[V(\theta_3,z)]\tforal z\in K_1(C).
$$
It follows that
$$
(\pi_0)_{*1}\circ \theta_3=(\pi_0)_{*1}\circ \theta_1={\mathrm{id}}_{K_1(C)}.
$$
Moreover, a dirct calculation shows that
$$
R_{\phi_1,\phi_2}\circ \theta_3=R_{\phi,\phi_2}\circ \theta_2-R_{\phi,\phi_1}\circ \theta_1=\rho_A\circ \dt.
$$
It follows from Theorem \ref{Lr0} that $\psi_1$ and $\psi_2$ are
asymptotically unitarily equivalent. Therefore, $\eta$ is one to
one.
%
\end{proof}


\begin{cor}\label{Cah}
Let $C$ be a unital separable amenable simple \CA\, with $\mathrm{TR}(C)=0$ which satisfies the UCT and let
$A$ be a unital simple \CA\, with $\mathrm{TR}(A)=0.$ Then the map ${\boldsymbol{K}}: \mathrm{Mon}_{asu}^e(C,A)\to {KK}_e(C,A)^{++}$ is surjective. Moreover,
the map
$$
\eta: \langle [\phi]\rangle \to \mathrm{Hom}(\Kone(C),\aff(\mathrm{T}(A)))/{\cal R}_0
$$
is bijective.
\end{cor}

\begin{proof}
It follows from \ref{MT2} that it suffices to point out that in this case $\rho_C({K}_0(C))$ is dense in $\aff(\mathrm{T}(C))$ and $\mathrm{T}_{\mathtt{f}}(C)=\mathrm{T}(C).$
\end{proof}

\begin{cor}\label{caut}
Let $A$ be a unital separable amenable simple \CA\, with $\mathrm{TR}(A)=0$ which satisfies the UCT.
Then the map ${\boldsymbol{K}}: \mathrm{Aut}(A)\to {KK}_e^{-1}(A,A)^{++}$ is surjective. Moreover,
the map
$$
\eta: \langle \rrm{id}_A\rangle \to \mathrm{Hom}(\Kone(A),\aff(\mathrm{T}(A)))/{\cal R}_0
$$
is bijective.
\end{cor}

\begin{proof}
This follows from \ref{Caut0} and \ref{rotation-maps} as in the proof of \ref{MT2}.
\end{proof}

\section{Classification of simple \CA s}\label{cla}

Denote by $Q$ the UHF algebra with $\Kzero(Q)=\Ratn$ with $[1_Q]=1$. If $\mathfrak{p}$ is a supernatural number, let $M_{\mathfrak{p}}$ denote the UHF-algebra associated with $\mathfrak{p}$.
\begin{lem}\label{Lasy}
Let $B$ be a unital separable amenable simple \CA\, such that $B\otimes Q$ has tracial rank zero.
Let $A$ be unital separable amenable  simple C*-algebras with tracial rank
zero satisfying the UCT, and let $\varphi_1, \varphi_2: A\to B$ be
two homomorphisms.
Suppose that $\phi, \psi: A\to B\otimes Q$ are two unital monomorphisms such that
$$
[\phi]=[\psi]\,\,\, {in}\,\,\,{KK}(A, B\otimes
Q).
$$
Suppose that $\phi$ induces an affine { homeomorphism}
${\phi_T}: \mathrm{T}(B\otimes Q)\to \mathrm{T}(A)$ by
$$
{\phi_T}(\tau)(a)=\tau\circ \phi(a)\tforal a\in A_{s.a}
$$
and for all $\tau\in \mathrm{T}(B\otimes Q).$
  Then there exists an automorphism $\alpha\in \mathrm{Aut}(\phi(A), \phi(A))$ with
  $[\alpha]=[\rrm{id}_{\phi(A)}]$ in ${KK}(\phi(A), \phi(A))$ such that
  $\alpha\circ \phi$ and $\psi$ are strongly asymptotically unitarily equivalent.
  \end{lem}

\begin{proof}
The proof is exactly the same as that of
the proof of Lemma 3.2 of \cite{Lin-App} but we apply \ref{rotation-maps} instead of 4.1 of \cite{KK} so that
the restriction on ${K}_i(A)$ ($i=0,1$) can be removed.
\end{proof}

\begin{defn}\label{Dsasy}
Let $C$ and $A$ be two unital \CA s. Let $\phi,\,\psi: C\to A$ be two homomorphisms. Recall that $\phi$ and $\psi$ are said to be strongly asymptotically unitarily equivalent  if there exists a continuous path of unitaries
$\{u(t): t\in [0, \infty)\}$ such that $u(0)=1$ and
$$
\lim_{t\to\infty}{\rrm{ad}}\, u(t)\circ \phi(c)=\psi(c)\tforal c\in C.
$$
\end{defn}

\begin{lem}[Theorem 3.4 of \cite{Lin-App}]\label{LT1}
Let $A$ and $B$ be two unital separable amenable simple \CA s
satisfying the UCT. Let ${\mathfrak{p}}$ and $\mathfrak{q}$ be
supernatural numbers of infinite type such that
$M_{\mathfrak{p}}\otimes M_{\mathfrak{q}}\cong Q.$ Suppose that
$A\otimes M_{\mathfrak{p}}, A\otimes M_{\mathfrak{q}},$ $B\otimes
M_{\mathfrak{p}}$ and $B\otimes M_{\mathfrak{q}}$ have tracial
rank zero.

Let $\sigma_{\mathfrak{p}}: A\otimes M_{\mathfrak{p}}\to B\otimes
M_{\mathfrak{p}}$ and $\rho_{\mathfrak{q}}: A\otimes
M_{\mathfrak{q}}\to B\otimes M_{\mathfrak{q}}$ be two unital
isomorphisms. Suppose
$$
[\sigma]=[\rho]
$$
in ${KK}(A\otimes Q, B\otimes Q),$
where $\sigma=\sigma_{\mathfrak{p}}\otimes \rrm
{id}_{M_{\mathfrak{q}}}$ and $\rho=\rho_{\mathfrak{q}}\otimes
{\rrm{id}}_{M_{\mathfrak{p}}}.$

Then there is an automorphism $\alpha\in
\mathrm{Aut}(\sigma_{\mathfrak{p}}(A\otimes M_{\mathfrak{p}}))$ such that
$$
[\alpha\circ \sigma_{\mathfrak{p}}]=[\sigma_{\mathfrak{p}}]
$$
in ${KK}(A\otimes M_{\mathfrak{p}}, B\otimes M_{\mathfrak{p}})$,
and $\alpha\circ \sigma_{\mathfrak{p}}\otimes {\rrm{
id}}_{M_{\mathfrak{q}}}$ is strongly asymptotically unitarily
equivalent to $\rho.$
\end{lem}

\begin{proof}
It follows from \ref{Lasy} that there exists $\beta\in
\mathrm{Aut}(B\otimes Q)$ such that $ \beta\circ \sigma $ is strongly
asymptotically unitarily equivalent to $\rho.$ Moreover,
$[\beta]=[{\rrm{id}}_{B\otimes Q}]$ in ${KK}(B\otimes Q, B\otimes
Q).$ Now consider two homomorphisms $\sigma_{\mathfrak{p}}$ and
$\beta\circ \sigma_{\mathfrak{p}}.$ One has
$$
[\beta\circ\sigma_{\mathfrak{p}}]=[\sigma_{\mathfrak{p}}]\,\,\,{\rrm
{in}}\,\,\,{KK}(A\otimes M_{\mathfrak{p}}, B\otimes Q).
$$
Since $\sigma_{\mathfrak{p}}$ is an isomorphism, it is easy to see
that ${\sigma_{T}}: \mathrm{T}(B\otimes Q)\to
\mathrm{T}(A\otimes M_{\mathfrak{p}})$ is an affine homeomorphism.

By applying \ref{Lasy} again, one obtains $\alpha\in
\mathrm{Aut}(\sigma_{\mathfrak{p}}(A\otimes M_{\mathfrak{p}}))$ such that
$[\alpha]=[\text{id}_{\sigma_{\mathfrak{p}}}]$ in
${KK}(\sigma_{\mathfrak{p}}(A\otimes
M_{\mathfrak{p}}),\sigma_{\mathfrak{p}}(A\otimes
M_{\mathfrak{p}}))$ and  $\alpha\circ\sigma_{\mathfrak{p}}$ is
strongly asymptotically unitarily equivalent to $\beta\circ
\sigma_{\mathfrak{p}}.$

Define $\beta\circ \sigma_{\mathfrak{p}}\otimes {\rrm
{id}}_{M_{\mathfrak{q}}}: A\otimes M_{\mathfrak{p}}\otimes
M_{\mathfrak{q}}\to (B\otimes Q)\otimes M_{\mathfrak{q}}.$ It is
easy to see that $\beta\circ \sigma_{\mathfrak{p}}\otimes {\rrm
{id}}_{M_{\mathfrak{q}}}$ is strongly approximately unitarily
equivalent to $\beta\circ \sigma.$

Note that $\sigma(A\otimes
M_{\mathfrak{p}})=B\otimes M_{\mathfrak{p}}.$ 
Let
$\sigma'=\alpha\circ \sigma_{\mathfrak{p}}\otimes
{\rrm{id}}_{M_{\mathfrak{q}}}.$
It follows that $\sigma'$ is strongly asymptotically unitarily
equivalent to $\beta\circ \sigma.$  Consequently $\sigma'$ is
strongly asymptotically unitarily equivalent to $\rho.$
\end{proof}

\begin{thm}\label{Class}
Let $A$ and $B$ be two unital separable amenable simple \CA s satisfying the UCT.
Suppose that there is an isomorphism
$$
\kappa: ({K}_0(A), {K}_0(A)_+, [1_A], {K}_1(A))\to ({K}_0(B), {K}_0(B)_+, [1_B], {K}_1(B)).
$$
Suppose also that
there is a pair of supper-natural numbers $\mathfrak{p}$ and $\mathfrak{q}$ of infinite type which are relative prime such that
$M_{\mathfrak{p}}\otimes M_{\mathfrak{q}}\cong Q$ and
$$
\mathrm{TR}(A\otimes M_{\mathfrak{p}})=\mathrm{TR}(A\otimes M_{\mathfrak{q}})=\mathrm{TR}(B\otimes M_{\mathfrak{p}})=\mathrm{TR}(B\otimes M_{\mathfrak{q}})=0.
$$
Then $A\otimes {\cal Z}\cong B\otimes {\cal Z}.$
\end{thm}

\begin{proof}
The proof is exactly the same as that of Theorem 3.5 of
\cite{Lin-App} but we now apply \ref{LT1} instead.
\end{proof}

\begin{cor}[Corollary 8.3 of \cite{Winter-Z}]\label{CM1}
Let $A$ and $B$ be two unital separable simple ASH-algebras
whose projections separate traces which are ${\cal Z}$-stable.
Suppose that
$$
({K}_0(A), {K}_0(A)_+, [1_A], {K}_1(A))\to ({K}_0(B), {K}_0(B)_+, [1_B], {K}_1(B)).
$$
Then $A\cong B.$
\end{cor}

\begin{proof}
As in the proof of 6.3 of \cite{Winter-Z}, $A\otimes C$ and $B\otimes C$
have tracial rank zero for any unital UHF-algebra $C$. Thus Theorem
\ref{Class} applies.

\end{proof}

\begin{cor}\label{CM2}
Let $A$ and $B$ be two unital separable simple ${\cal Z}$-stable \CA s which are inductive limits of type $I$ \CA s with unique tracial states. Suppose that
$$
({K}_0(A), {K}_0(A)_+, [1_A], {K}_1(A))\to ({K}_0(B), {K}_0(B)_+, [1_B], {K}_1(B)).
$$
Then $A\cong B.$
\end{cor}

\begin{proof}
For any UHF-algebra $C,$ $A\otimes C$ is approximately divisible. Since $A$ has a unique tracial state, so does $A\otimes C.$ Therefore projections of $A\otimes C$ separate traces. It follows from \cite{RorUHF} that $A\otimes C$ has real rank zero, stable rank one and weakly unperforated $\Kzero(A\otimes C).$ Moreover, $A\otimes C$ is also an inductive limit of type I C*-algebras. It follows from Theorem 4.15 and Proposition 5.4 of \cite{Lin-Corelle} that $\mathrm{TR}(A\otimes C)=0.$ Exactly the same argument shows that $\mathrm{TR}(B\otimes C)=0.$ It follows from Theorem \ref{Class} that $A\cong B.$
\end{proof}

\bibliographystyle{plain}

\providecommand{\bysame}{\leavevmode\hbox to3em{\hrulefill}\thinspace}
\providecommand{\MR}{\relax\ifhmode\unskip\space\fi MR }
\providecommand{\MRhref}[2]{%
  \href{http://www.ams.org/mathscinet-getitem?mr=#1}{#2}
}
\providecommand{\href}[2]{#2}

\end{document}